\documentclass[a4paper,10pt,twoside,reqno]{amsart}

\usepackage[utf8]{inputenc}
\usepackage{amsmath, amsfonts, amssymb,amsthm, amscd}
\usepackage{mathtools}
\usepackage[mathscr]{eucal}
\usepackage[pdftex]{graphicx}
\graphicspath{{img/}}
\usepackage{color}
\usepackage{multicol}
\usepackage[colorlinks=false, linktocpage=false, pdfborder={0 0 0}, pdfstartpage=3, pdfstartview=FitV] {hyperref}

\def\equationautorefname~#1\null{(#1)\null}

\usepackage[T1]{fontenc}
\usepackage[sc, osf]{mathpazo}
\usepackage{enumerate}

\usepackage{xypic}

\newcommand{\bR}{\overline{\mathrm{R}}} 

\DeclareMathOperator{\diver}{div} 
\DeclareMathOperator{\Ric}{Ric} 
\DeclareMathOperator{\Ind}{Ind} 
\DeclareMathOperator{\Nul}{Nul} 
\DeclareMathOperator{\trace}{tr} 
\DeclareMathOperator{\pIm}{Im} 
\DeclareMathOperator{\Span}{span}  

\newcommand{\bRic}{\overline{\Ric}} 

\newcommand{\prodesc}[2]{g(#1,#2)} 
\newcommand{\bprodesc}[2]{\langle#1,#2\rangle} 

\newcommand{\nablab}{\overline{\nabla}} 
\newcommand{\nablap}{\overline{D}} 
\newcommand{\nablar}{\prescript{g}{}{\overline{\nabla}}} 
\newcommand{\nablae}{\prescript{g}{}{\overline{D}}} 
\newcommand{\nablag}{\prescript{g}{}{\nabla}} 

\newcommand{\Deltab}{\bar{\Delta}} 
\newcommand{\Deltar}{\prescript{g}{}{\bar{\Delta}}} 
\newcommand{\HM}{HM}

\newtheorem{theorem}{Theorem}
\newtheorem{proposition}{Proposition}
\newtheorem{corollary}[theorem]{Corollary}
\newtheorem{lemma}{Lemma}

\theoremstyle{definition}

\newtheorem{example}{Example}

\theoremstyle{remark}
\newtheorem{remark}{Remark}

\numberwithin{equation}{section}

\title{Index of compact minimal submanifolds of the Berger spheres}

\author{Francisco Torralbo}
\address{Departamento de Geometría y Topología, Universidad de Granada, Spain}
\email{ftorralbo@ugr.es}

\author{Francisco Urbano}
\address{Departamento de Geometría y Topología, Universidad de Granada, Spain}
\email{furbano@ugr.es}

\thanks{
  The first author is supported by project \textsc{pid2019.111531ga.i00} funded by \textsc{mcin/ aei/10.13039/501100011033} and the Programa Operativo \textsc{feder} Andalucía 2014-2020, grant no.\ \textsc{e-fqm-309-ugr18}. The second author is supported by Regional J. Andalucía grant no. \textsc{p18-fr-4049}. 
}



\begin{document}

\title{Index of compact minimal submanifolds of the Berger spheres}

\begin{abstract}

The stability and the index of compact minimal submanifolds of the Berger spheres $\mathbb{S}^{2n+1}_{\tau},\, 0<\tau\leq 1$,  are studied. Unlike the case of the standard sphere ($\tau=1$), where there are no stable compact minimal submanifolds, the Berger spheres have stable ones  if and only if $\tau^2\leq 1/2$. Moreover, there are no stable compact minimal $d$-dimensional submanifolds of $\mathbb{S}^{2n+1}_\tau$ when $1 / (d+1) < \tau^2 \leq 1$ and the stable ones are classified for $\tau^2=1 / (d+1)$ when the submanifold is embedded. Finally, the compact orientable minimal surfaces of $\mathbb{S}^3_{\tau}$ with index one are classified for $1/3\leq\tau^2\leq 1$.
\end{abstract}

\maketitle

\subsection*{Acknowledgements}

The first author is supported by project PID2019.111531GA.I00 funded by MCIN/AEI/10.13039/501100011033 and the Programa Operativo FEDER Andalucía 2014-2020, grant no.~E-FQM-309-UGR18. The second author is supported by Regional J. Andalucía grant no. P18-FR-4049. 

The authors would like to thank the referees for thorough revisions of the manuscript and inestimable suggestions to improve it.

Data sharing not applicable to this article as no datasets were generated or analysed during the current study.

\section{Introduction}

The second variation operator of minimal submanifolds of Riemannian manifolds (the Jacobi operator) carries information about stability properties of the submanifold when it is thought of as a critical point of the volume functional. Perhaps the starting point is  the paper of J. Simons~\cite{Simons68}, where he characterized the compact minimal submanifolds of the sphere with the lowest index (number of independent infinitesimal deformations which do decrease the volume), proving that there are no stable ones. Later, H.~B.~Lawson and J.~Simons~\cite{LS73}, and Y.~Onhita~\cite{Onhita86}, studied similar problems when the ambient Riemannian manifold is a compact rank-one symmetric space.

An important particular case in this setting is when the submanifold is a two-sided hypersurface. In this case,  the normal bundle is trivial and of rank one. Then the Jacobi operator, which acts on the sections of the normal bundle, becomes a Schrödinger operator acting on functions. In this case the study of the index of the Jacobi operator has been made for complete (not necessarily compact) minimal hypersurfaces, and the results of D. Fischer-Colbrie and R. Schoen, \cite{Fischer-Colbrie1985, FCS1980}, have been fundamental in the growth of this theory.

During the last forty years many papers have been devoted to study stability and index of minimal submanifolds in different ambient Riemannian manifolds. Among them, we only mention two of the last ones. First, the recent paper of O. Chodosh and D. Maximo~\cite{CM2021}, where they get a lower bound of the index of complete minimal surfaces in $\mathbb{R}^3$ in terms of the genus, the number of ends and the multiplicity of the surfaces. Second, the paper of F. C. Marques and A.\ Neves, \cite{MN2020} and references therein, where  given a compact manifold $M^{n}$, $3\leq n\leq 7$ with a generic metric, they use min-max theory  to construct for each positive integer number $k$  a two-sided embedded minimal hypersurface of $M$ with index $k$.

In the present work, we are interested in the index of compact minimal submanifolds of the Berger spheres. If $\mathbb{S}^{2n+1}$ is the unit sphere of dimension $2n+1$, the Berger spheres are a $1$-parameter family $\{(\mathbb{S}^{2n+1},\langle.\rangle_\tau)\colon 0<\tau\leq 1\}$, where $\langle,\rangle_{1}$ is the standard metric $g$ on  $\mathbb{S}^{2n+1}$ and the metric $\langle,\rangle_{\tau}$ is given in ~\eqref{eq:defini-metrica Berger}. These Berger spheres are often called elliptic Berger spheres. This deformation of the 
standard metric $g$ is a classical example of a collapse~\cite{CG1990}, in which the injectivity radius $i_{\tau}$ of $\langle,\rangle_{\tau}$ is $\tau\pi$ and has uniformly bounded positive curvature.  This deformation of the standard metric $g$ can be defined even for $\tau\in (1,\infty)$, but the geometry of these Berger spheres, called hyperbolic, is quite different from the elliptic ones. In the paper, although some minor results are also true for $\tau\in (1,\infty)$, we only consider Berger spheres with $\tau\in (0,1]$.

The standard sphere $\mathbb{S}^{2n+1}$ is a geodesic sphere of $\mathbb{C}^{n+1}$ of radius $1$. In a similar way,  the Berger sphere $\mathbb{S}^{2n+1}_{\tau}$, $\tau\in (0,1)$,  is isometric to a
geodesic sphere of the complex projective space $\mathbb{CP}^{n+1}(4(1-\tau^2))$ of radius $\frac{\arccos(\tau)}{\sqrt{1-\tau^2}}$ and the restrictions to $\mathbb{S}^{2n+1}_{\tau}$ of the complex structures of $\mathbb{C}^{n+1}$ and $\mathbb{CP}^{n+1}(4(1-\tau^2))$ are the same (see \autoref{prop:geodesic-spheres-complex-spaces}). Along the paper, any minimal submanifold of $\mathbb{S}^{2n+1}_{\tau}$ will be considered as a submanifold of $\mathbb{CP}^{n+1}(4(1-\tau^2))$ and its behaviour with respect to the complex structure will play an important role in the proofs of the results.

On the other hand, the Hopf fibration $\pi:\mathbb{S}^{2n+1}\rightarrow\mathbb{CP}^n(4)$ defines a Riemannian submersion $\pi:\mathbb{S}^{2n+1}_{\tau}\rightarrow\mathbb{CP}^n(4)$ for any $\tau\in (0,1]$. The induced $\mathbb{S}^1$-bundles by the Hopf fibration over  compact complex submanifolds of $\mathbb{CP}^n(4)$ (see \autoref{ex:induced-bundle}) define  examples of compact minimal submanifolds of $\mathbb{S}^{2n+1}_{\tau}$ which have a very good behaviour with respect to the second variation of the volume. These kind of minimal submanifolds will provide the more regular examples, some of which will be characterized by their index.

In \autoref{sec:minimal-submanifolds-Berger-spheres}, we introduce some important examples of compact minimal submanifolds of $\mathbb{S}^{2n+1}_{\tau}$, classifying  in \autoref{prop:classification-totally-geodesic} the totally geodesic ones.  We compute the index and  nullity of the above examples in Propositions~\ref{prop:index-nullity-totally-geodesic-Berger-spheres}, \ref{prop:index-nullity-Veronese}, \ref{prop:index-nullity-totally-geodesic-spheres} and~\ref{prop:Clifford-hypersurfaces-index-nullity}, and, among them, stable examples appear.  The computation of the index and the nullity of these examples is not easy and, to do that, it has been crucial the paper of S. Tanno, \cite{Tanno1979}, in which he gives important insight into the spectrum of the Laplacian of $\mathbb{S}^{2n+1}_{\tau}$.

 The main results about stability of  compact minimal submanifolds of $\mathbb{S}^{2n+1}_{\tau}$ are obtained in  Theorems~\ref{thm:S1-bundle-are-stable} and~\ref{thm:classification-stable} and they can be  summarized  as follows:
\begin{quote}
{\it There are no stable immersed compact minimal $d$-dimensional submanifolds  of $\mathbb{S}^{2n+1}_{\tau}$ when $\frac{1}{d+1}<\tau^2\leq 1$.}
\end{quote}
\begin{quote}
{\it If $\tau^2=\frac{1}{d+1}$ for some positive integer $d$, then an embedded compact minimal submanifold $M^d$ of $\mathbb{S}^{2n+1}_{\tau}$ is  stable if and only if $d=2m+1$ and  $M^{2m+1}$  is the induced $\mathbb{S}^1$-bundle by the Hopf fibration over  an  embedded compact complex submanifold $N^{2m}$ of $\mathbb{CP}^n(4)$.}
\end{quote}

\begin{quote}
{\it If $0<\tau^2\leq \frac{1}{2m+2}$ for some integer $m \geq 0$, then any immersed submanifold $M^{2m+1}$ which is the induced $\mathbb{S}^1$-bundle by the Hopf fibration $\pi: \mathbb{S}^{2n+1}_\tau \to \mathbb{CP}^n(4)$, over a compact complex submanifold $N^{2m}$ of $\mathbb{CP}^n(4)$ is stable. }
\end{quote}

In the proof of the two first above  results, we use test normal sections on the second variation coming from parallel vector fields of the complex  Euclidean space $\mathbb{C}^{n+1}$ and  from holomorphic vector fields of the complex projective space $\mathbb{CP}^{n+1}(4(1-\tau^2))$.

When $M^{2n}$ is a  compact orientable minimal hypersurface of $\mathbb{S}^{2n+1}_{\tau}$, then $M$ is unstable (\autoref{prop:hypersurface-unstable}). Hence, in this case, is natural to study the hypersurfaces of index $1$. From \autoref{prop:Clifford-hypersurfaces-index-nullity}, we know that the Clifford hypersurfaces have index $1$ when $0<\tau^2\leq \frac{1}{2n+1}$. In \autoref{thm:characterization-index-surfaces}, we give an answer to this problem when $M$ is a surface, proving:
\begin{quote}
{\it If $\frac{1}{3}\leq\tau^2\leq 1$, then a compact orientable minimal surface of $\mathbb{S}^3_{\tau}$ has index one if and only if $M$ is either the unique minimal sphere of $\mathbb{S}^3_{\tau}$ or the  Clifford surface in $\mathbb{S}^3_{1/\sqrt{3}}$.}
\end{quote}

For the proof of this result, we embed isometrically $\mathbb{CP}^2(4(1-\tau^2))$ in the Euclidean space (see appendix),  so we can see our surface isometrically immersed in a  sphere $\mathbb{S}^{7}(r)$, for certain radius $r$. In this case, we use as test functions on the second variation the restriction to the surface of the components of certain conformal transformations of $\mathbb{S}^7(r)$.

\section{The Berger spheres}

The Berger spheres $\mathbb{S}^{2n+1}_\tau,\, 0<\tau\leq 1,$  are the Riemannian  manifolds $(\mathbb{S}^{2n+1}, \langle\cdot,\cdot\rangle_{\tau})$ where $\langle\cdot,\cdot\rangle_{\tau}$ is the Riemannian metric on the unit sphere $\mathbb{S}^{2n+1} = \{z \in \mathbb{C}^{n+1}\colon \lvert z\rvert^2 = 1\}$ given by 
\begin{equation}\label{eq:defini-metrica Berger}
\langle v, w\rangle_\tau = \prodesc{v}{w} - (1-\tau^2)\prodesc{v}{iz} \prodesc{w}{iz}, \quad v, w \in T_z \mathbb{S}^{2n-1},
\end{equation} 
where $g$ is the Euclidean metric in $\mathbb{C}^{n+1}$ and $i$ is the imaginary unit. Notice that $\langle{\cdot},{\cdot}\rangle_1 = g$. In the sequel we will denote the Berger metric simply by $\bprodesc{\cdot}{\cdot}$ omitting the subindex $\tau$. 

The isometry group of $\mathbb{S}^{2n+1}_\tau$, $0 < \tau < 1$, is the subgroup of the orthogonal group $O(2n+2)$ given by $\{A \in O(2n+2)\colon A\mathcal{J} = \pm \mathcal{J}A\}$ where $\mathcal{J} \in O(2n+2)$ is the matrix associated with the multiplication by $i$ in $\mathbb{C}^{n+1}$. The subgroup $\{A \in O(2n+2)\colon A\mathcal{J} = \mathcal{J}A\}$ is the real representation of the unitary group $U(n+1)$ and hence the dimension of the isometry group of $\mathbb{S}^{2n+1}_\tau$ is $(n+1)^2$.  In this context, it is clear that the group $U(n+1)$ acts transitively  on $\mathbb{S}^{2n+1}_\tau$ and the isotropy subgroup at a point of $\mathbb{S}^{2n+1}_\tau$ is isomorphic to $U(n)$. Hence the Berger sphere is a homogeneous Riemannian manifold diffeomorphic to $U(n+1) / U(n)$.

On the other hand, the Hopf fibration $\pi: \mathbb{S}^{2n+1} \to \mathbb{C}\mathbb{P}^n(4)$, where $\mathbb{C}\mathbb{P}^n(4)$ denotes the complex projective space endowed with the Fubini-Study metric of constant holomorphic sectional curvature $4$, is a Riemannian submersion from $\mathbb{S}^{2n+1}_{\tau}$ onto $\mathbb{CP}^n(4)$,  whose fibers are circles of length $2\pi \tau$. The unit vector field $\xi$ on $\mathbb{S}^{2n+1}_{\tau}$, defined by  
\begin{equation}\label{eq:killing}
\xi_z = \tfrac{1}{\tau}iz,\quad\forall z\in\mathbb{S}^{2n+1}_{\tau},
\end{equation}
spans the vertical line of the Hopf fibration and its uniparametric group is given by
\begin{equation}\label{eq:uniparametric}
\zeta(t,z)=\cos(\tfrac{t}{\tau})z+\sin (\tfrac{t}{\tau})iz,\quad t\in\mathbb{R}, \, z\in\mathbb{S}^{2n+1}_{\tau}.
\end{equation}  
Hence the action of $\mathbb{S}^1$ over $\mathbb{S}^{2n+1}$  can be written as $e^{it} \cdot z = \zeta(\tau t,z)$, $0\leq t\leq 2\pi$. Moreover, for any $z\in\mathbb{S}^{2n+1}_{\tau}$, the tangent space $T_z \mathbb{S}^{2n+1}_\tau$ can be orthogonally decomposed as $\langle \xi_z \rangle \oplus \mathcal{H}_z$, where $\mathrm{d}\pi_z:\mathcal{H}_z\rightarrow T_{\pi(z)}\mathbb{CP}^n(4)$ is a linear isometry. We note that the metrics $\bprodesc{\cdot}{\cdot}$ and $g$ coincide on $\mathcal{H}$, and $iv \in \mathcal{H}_z$ for any $v \in \mathcal{H}_z$.

Along the paper we will see the Berger spheres $\mathbb{S}^{2n+1}_{\tau}$ as geodesic spheres of the complex projective space $\mathbb{CP}^{n+1}(4(1-\tau^2))$. The following result describes it explicitly.

\begin{proposition}\label{prop:geodesic-spheres-complex-spaces}
  The map $F: \mathbb{S}^{2n+1}_\tau \to \mathbb{CP}^{n+1}(4(1-\tau^2))$ given by 
\[
F(z_1,\ldots, z_{n+1}) = \left[ \big(\tfrac{\tau}{\sqrt{1-\tau^2}}, z_1,\ldots, z_{n+1}\big)\right],
\] 
where $[w]=\pi(w)$ for any $w\in\mathbb{S}^{2n+3}(\frac{1}{\sqrt{1-\tau^2}})$ and $\pi:\mathbb{S}^{2n+3}(\frac{1}{\sqrt{1-\tau^2}})\rightarrow\mathbb{CP}^{n+1}(4(1-\tau^2)$ is the canonical projection,  
is an isometric embedding of the Berger sphere  $\mathbb{S}^{2n+1}_{\tau}$ into the complex projective space of constant holomorphic sectional curvature $4(1-\tau^2)$. Hence, the Berger sphere $\mathbb{S}^{2n+1}_\tau$ can be identified with the set%
\[
  \mathbb{S}^{2n+1}_\tau \equiv \left\{[(w_0,w_1,\ldots,w_{n+1})] \in \mathbb{CP}^{n+1}(4(1-\tau^2)) \colon \lvert w_0\rvert^2 = \tfrac{\tau^2}{1-\tau^2}\right\},
\] 
which is a geodesic sphere of $\mathbb{CP}^{n+1}(4(1-\tau^2))$ with radius $\frac{1}{\sqrt{1-\tau^2}} \arccos(\tau)$ and center $[(\frac{1}{\sqrt{1-\tau^2}},0,\ldots,0)]$. The Fubini-Study metric of $\mathbb{CP}^{n+1}(4(1-\tau^2))$ will be also denoted by $\bprodesc{\cdot}{\cdot}$.

Moreover, if $J$ denotes the complex structure of $\mathbb{CP}^{n+1}(4(1-\tau^2))$, then  $ JdF_z(
\xi_z)$ is a unit normal vector  to $\mathbb{S}^{2n+1}_{\tau}$ in $\mathbb{CP}^{n+1}(4(1-\tau^2))$, and $dF_z(iu) = JdF_z(u)$ for any horizontal vector $u \in T_z\mathbb{S}^{2n+1}_{\tau}$. 

Finally, the second fundamental form of the embedding $F$ is given by
  \begin{equation} \label{eq:2ff-Berger-complex-space}
    \hat{\sigma}(X, Y) = \left(\tau \bprodesc{X}{Y}  - \tfrac{1-\tau^2}{\tau}\bprodesc{X}{\xi}\bprodesc{Y}{\xi}\right) J \xi,\quad\forall X,Y\in\mathfrak{X}(\mathbb{S}^{2n+1}_{\tau}). \\
  \end{equation}
\end{proposition}

\begin{proof}
Firstly, we can write $F = \pi \circ \hat{F}$, where $\hat{F}:\mathbb{S}^{2n+1}_{\tau}\rightarrow\mathbb{S}^{2n+3}(\frac{1}{\sqrt{1-\tau^2}})$ is given by $\hat{F}(z)=\Bigl(\frac{\tau}{\sqrt{(1-\tau^2)}},z\Bigr)$, for any $z\in\mathbb{S}^{2n+1}_{\tau}$.

 Then, for any $u\in T_z\mathbb{S}^{2n+1}_{\tau}$ we have $\mathrm{d}\hat{F}_z(u)=(0,u)$ and hence its horizontal component with respect to $\pi$ is $(0,u)^H=(0,u)-(1-\tau^2)g(u,iz)i\hat{F}(z)$. From here, and taking into account \eqref{eq:defini-metrica Berger}, it follows that
\begin{equation}\label{eq:dF-isometry-Berger-complex-projective-space}
  \langle\mathrm{d}F_z(u),\mathrm{d}F_z(v)\rangle  = g(\mathrm{d}\hat{F}_z(u)^H, \mathrm{d}\hat{F}_z(v)^H ) = \langle u,v\rangle,\quad \forall u,v\in T_z\mathbb{S}^{2n+1}_{\tau}, 
\end{equation} 
which means that $F$ is an isometric immersion. As clearly $F$ is injective, we obtain that $F$ is an isometric embedding.

 Also, if $u \in  T_z \mathbb{S}^{2n+1}_\tau$ is a horizontal vector, that is, $\prodesc{u}{iz} = 0$, we have, from the definition of the complex structure $J$ of $\mathbb{CP}^{n+1}(4(1-\tau^2))$, that
 \[
 J \mathrm{d}F_z(u)=\mathrm{d}\pi_{\hat{F}(z)}(i(0,u)^H)=\mathrm{d}\pi_{\hat{F}(z)}((0,iu))=\mathrm{d}F_z(iu).
 \]
Moreover, using~\eqref{eq:killing} we obtain that 
\begin{equation}\label{eq:dF-xi-isometry-Berger-complex-projective-space}
  J\mathrm{d}F_z (\xi_z) = \mathrm{d}\pi_{\hat{F}(z)}(i(0,\xi_z)^H) = \mathrm{d}\pi_{\hat{F}(z)}(\sqrt{1-\tau^2},-\tau z).
\end{equation} 
So, for every $u\in T_z \mathbb{S}^{2n+1}_\tau$, we have  \[
  \bprodesc{\mathrm{d}F_z(u)}{J \mathrm{d}F_z(\xi_z)}  =\prodesc{\mathrm{d}\hat{F}_z(u)^H}{(\sqrt{1-\tau^2},-\tau z)}= 0,
\]
 and hence  $z\in\mathbb{S}^{2n+1}_{\tau}\mapsto J \mathrm{d}F_z(\xi_z)$ defines a unit normal vector field to $F$.

 Finally, it is clear that $F(\mathbb{S}^{2n+1}_{\tau})$ is a geodesic sphere of $\mathbb{CP}^{n+1}(4(1-\tau^2))$ of radius $\frac{1}{\sqrt{1-\tau^2}} \arccos(\tau)$ and hence it is a pseudoumbilical hypersurface with two constant principal curvatures: $\tau$ with multiplicity $2n$ and $\frac{2\tau^2-1}{\tau}$ with multiplicity $1$, see~\cite{Takagi75}. As $\xi\equiv F_*\xi$ is an eigenvector for the eigenvalue $\frac{2\tau^2-1}{\tau}$,  the second fundamental form $\hat{\sigma}$ of the embedding $F$ is given by ~\eqref{eq:2ff-Berger-complex-space}.
\end{proof}

 Taking into account \autoref{prop:geodesic-spheres-complex-spaces} and the expression for the curvature tensor in the complex projective space, we get that the curvature tensor $\bR$  of $\mathbb{S}^{2n+1}_\tau$ is given by
  \begin{multline}\label{eq:Riemann-tensor-Berger}
    \bR(X, Y, Z, W) = \bprodesc{Y}{Z} \bprodesc{X}{W} - \bprodesc{X}{Z} \bprodesc{Y}{W} \\
        + (1-\tau^2)\bigl[\bprodesc{J Y}{Z} \bprodesc{JX}{W} - \bprodesc{J X}{Z} \bprodesc{JY}{W} - 2\bprodesc{J X}{Y} \bprodesc{JZ}{W}\bigr] \\
        + (1-\tau^2) \bprodesc{Z}{\xi}\bigl[ \bprodesc{X}{\xi} \bprodesc{Y}{W} - \bprodesc{Y}{\xi} \bprodesc{X}{W}\bigr] \\
        + (1-\tau^2) \bprodesc{W}{\xi} \bigl[ \bprodesc{Y}{\xi}\bprodesc{X}{Z} - \bprodesc{X}{\xi}\bprodesc{Y}{Z}\bigr],
  \end{multline}
for every $X,Y,Z,W\in\mathfrak{X}(\mathbb{S}^{2n+1}_{\tau})$.

In particular, the sectional curvature of the plane $\Pi\subset T_p\mathbb{S}^{2n+1}_{\tau}$ generated by an orthonormal frame $\{v, w\}$ is given by
\begin{equation}\label{eq:sectional-curvature}
    K(\Pi) = 1 + (1-\tau^2)\bigl[3 \bprodesc{v}{Jw}^2 - \lvert\xi^\Pi\rvert^2\bigr],
\end{equation}
where $()^\Pi$ denotes the tangent component to $\Pi$. Moreover, the Ricci curvature $\bRic$ of a unit vector $v$ is
\begin{equation}\label{eq:ricci}
    \bRic(v) = 2n + 2(1-\tau^2) [1 - (n+1) \bprodesc{v}{\xi}^2], \\
\end{equation}
and the scalar curvature is $\overline\rho=2n[2(n+1)-\tau^2]$.
It is interesting to note that  $\bRic \geq 2n\tau^2> 0$ . 

In the next result we relate some geometric objects of the metrics $\bprodesc{\cdot}{\cdot}$ and $\prodesc{\cdot}{\cdot}$ on the sphere $\mathbb{S}^{2n+1}$.

\begin{lemma}\label{lm:connection-Berger}
  Let $\nablab$ and $\nablar$ be the Levi-Civita connections of the Berger metric $\bprodesc{\cdot}{\cdot}$ and the standard metric $\prodesc{\cdot}{\cdot}$ on $\mathbb{S}^{2n+1}$ respectively. Then:
	\begin{enumerate}
    \item For any $X, Y \in \mathfrak{X}(\mathbb{S}^{2n+1})$ 
	\begin{equation}\label{eq:Levi-Civita-connection-Berger-round}
	\nablar_X Y=	\nablab_X Y + \tfrac{1 - \tau^2}{\tau}\bigl[ \bprodesc{Y}{\xi} J(X - \bprodesc{X}{\xi}\xi) + \bprodesc{X}{\xi} J(Y - \bprodesc{Y}{\xi}\xi) \bigr],
	\end{equation}
  where $\xi$ is the vector field defined in \eqref{eq:killing}.
\item For any $X\in \mathfrak{X}(\mathbb{S}^{2n+1}_\tau)$,
 \begin{equation}\label{eq:covariant-derivative-xi}
  \nablab_X \xi = \tau J(X - \bprodesc{X}{\xi}\xi),
\end{equation}
So  $\xi$ is a unit geodesic Killing field on $\mathbb{S}^{2n+1}_{\tau}$.
\item The gradients and the Laplacians of a function $f:\mathbb{S}^{2n+1}\rightarrow \mathbb{R}$ with respect to both metrics are related by
\begin{equation}\label{eq:laplacian}
	\begin{aligned}
		\nablar f&=\nablab f-(1-\tau^2)(L_{\xi}f)\xi,\\
		\Deltar f &= \Deltab f-(1-\tau^2)(L_{\xi})^2 f, 
	\end{aligned}
\end{equation}
where $L_{\xi}$ is the Lie derivative with respect to the vector field $\xi$.
\item The spectrum of the Laplacian operator $\Deltab$ of $\mathbb{S}^{2n+1}_{\tau}$ is contained in the set
\[
  \left\{\mu_{k,p}=\lambda_k+\tfrac{1-\tau^2}{\tau^2}(k-2p)^2\colon k\in\mathbb{Z}^+,\,p\in\mathbb{Z},\, 0\leq p\leq\lfloor\tfrac{k}{2}\rfloor\right\},
\]
where $\{\lambda_k=k(2n+k),\, k\in\mathbb{Z}^+\}$ is the spectrum of $(\mathbb{S}^{2n+1},g)$ and $\lfloor\cdot\rfloor$ stands for integer part. Moreover, $\Deltar\circ L_{\xi}=L_{\xi}\circ \Deltar$ and the eigenspace $V(\lambda_k)$ of the eigenvalue $\lambda_k$ decomposes as
\[
  V(\lambda_k)=V(\mu_{k,0})\oplus \dots\oplus V(\mu_{k,[\frac{k}{2}]}), 
 \]
where some of $V(\mu_{k,p})$ may be trivial. Moreover, for each $\varphi \in V(\mu_{k,p})$ we have $(L_\xi)^2 \varphi + \frac{1}{\tau^2}(k - 2p)^2 \varphi = 0$.
\end{enumerate}
\end{lemma}

\begin{remark}\label{rmk:Eigenspaces-Tanno}
  For $n = 0$, the metrics $\langle,\rangle_{\tau}$ and $g$ are homothetic on $\mathbb{S}^1$. Hence,  we have that $V(\lambda_k) = V(\mu_{k,0})$ and so $V(\mu_{k,p})$ are trivial for all $p \neq 0$.  For $n \geq 1$, Tanno~\cite{Tanno1968} shows that $V(\mu_{k,0})$, $V(\mu_{k,k/2})$ and $V(\mu_{k,(k-2)/2})$ ($k$ even), and $V(\mu_{k, (k-1)/2})$ ($k$ odd) are always non-trivial.
\end{remark}

\begin{proof}
Using the expression of the Levi-Civita connection given in the Koszul formula and the relation between the metrics $\bprodesc{\cdot}{\cdot}$ and $\prodesc{\cdot}{\cdot}$ given in ~\eqref{eq:defini-metrica Berger}, it is straightforward  to get \eqref{eq:Levi-Civita-connection-Berger-round}. Now using that $\nablar_Xiz$ is the tangential component to the sphere of $iX$ and \eqref{eq:Levi-Civita-connection-Berger-round} we easily get \eqref{eq:covariant-derivative-xi}.

If $f$ is a smooth function on $\mathbb{S}^{2n+1}$ then  $X(f)=\bprodesc{\nablab f}{X}=\prodesc{\nablar f}{X}$. Now, using \eqref{eq:defini-metrica Berger} we get the relations between the gradients given in \eqref{eq:laplacian}. Under this condition, the relations between the Laplacians given in \eqref{eq:laplacian} is a direct consequence of \eqref{eq:Levi-Civita-connection-Berger-round}. Finally, item (4) was proved by S.\ Tanno~\cite{Tanno1979}.  
\end{proof}

\section{Minimal submanifolds in the Berger spheres}\label{sec:minimal-submanifolds-Berger-spheres}

Let $\Phi:M^d\rightarrow\mathbb{S}^{2n+1}_{\tau}$ be an  immersion of a compact $d$-dimensional manifold $M$ in the Berger sphere $\mathbb{S}^{2n+1}_\tau$. Then, $\Phi^*T\mathbb{S}^{2n+1}_{\tau}=TM\oplus T^{\perp}M$, where $T^{\perp}M$ is the normal bundle of $\Phi$. Let $\nablab$ be the Levi-Civita connection of $\mathbb{S}^{2n+1}_\tau$, $\nabla$ be the Levi-Civita connection of the induced metric on $M$, and $\nabla^\perp$ be the normal connection.

Suppose now that $\Phi$ is a minimal immersion, that is $\Phi$ is a critical point for the volume functional. The second variation operator, known as the {\it Jacobi operator} of $\Phi$ and which will be denote by $\mathcal{L}$, is a strongly elliptic operator acting on the sections of the normal bundle, $\mathcal{L}:\mathfrak{X}^{\perp}(M)\rightarrow\mathfrak{X}^{\perp}(M)$ given by
\[ 
\mathcal{L}  = \Delta^\perp  + \mathcal{A} + \mathcal{R},
\]
where $\Delta^{\perp}$ is the normal Laplacian
\[
\Delta^\perp  = \sum_{i = 1}^{d} \bigl( \nabla^\perp_{e_i} \nabla^\perp_{e_i}  - \nabla^\perp_{\nabla_{e_i}e_i}\bigr),
\]
and  $\mathcal{A}, \mathcal{R}$ are the endomorphisms of the normal bundle defined as follows
\[
\mathcal{A}(\eta)=\sum_{i=1}^d\sigma(e_i,A_{\eta}e_i),\quad \mathcal{R}(\eta)=\sum_{i=1}^d(\bR(\eta,e_i)e_i)^{\perp},
\]
where $\{e_1,\dots,e_d\}$ is an orthonormal reference tangent to $M$, $\sigma$ is the second fundamental form of $\Phi$, $A_{\eta}$ is the Weingarten endomorphism associated to $\eta\in \mathfrak{X}^{\perp}(M)$ and $\perp$ stands for normal component. Let $\mathcal{Q}:\mathfrak{X}^{\perp}(M)\rightarrow\mathbb{R}$ be the quadratic form associated to the Jacobi operator $\mathcal{L}$, defined by
\[
  \mathcal{Q}(\eta)=-\int_M \bprodesc{\mathcal{L} \eta}{\eta}\, \mathrm{d}v.
\]
 We will denote by $\Ind(\Phi)$ and $\Nul(\Phi)$ (they will be also denoted by $\Ind(M)$ and $\Nul(M)$) the {\it index} and the {\it nullity} of the quadratic form $\mathcal{Q}$, which are respectively the number of negative eigenvalues of $\mathcal{L}$ and the multiplicity of zero as an eigenvalue of $\mathcal{L}$. The immersion $\Phi$ will be called {\it stable} if $\Ind(\Phi)=0$.

Using~\eqref{eq:Riemann-tensor-Berger} we obtain that
\begin{equation}\label{eq:Jacobi-operator}
  \mathcal{L} \eta = \Delta^\perp \eta + \mathcal{A}\eta + \big(d - (1-\tau^2)\lvert\xi^\top\rvert^2\big)\eta-d(1-\tau^2) \bprodesc{\eta}{\xi}\xi^\perp -  3(1-\tau^2) [J(J\eta)^\top]^\perp,
\end{equation} 
where $\top$ and $\perp$ denote respectively  tangential and normal components to $M$.

In the particular case $M$ is an orientable hypersurface, i.e.\ $d = 2n$, and $N$ is a unit normal vector field to $\Phi$, any normal section $\eta$ can be written $\eta=fN$ and so $\mathfrak{X}^{\perp}(M)\equiv \mathcal{C}^{\infty}(M)$ ($\eta \equiv f $).  The operator  $\mathcal{L}$ becomes in the Schrödinger operator $\mathcal{L}:\mathcal{C}^{\infty}(M)\rightarrow \mathcal{C}^{\infty}(M)$  given by 
\begin{equation}\label{eq:Jacobi-operator-hypersurfaces}
  \mathcal{L}  = \Delta + \lvert\sigma\rvert^2 + 2n - 2(1-\tau^2) \bigl((n+1) \nu^2 - 1\bigr) ,
\end{equation}
where $\Delta$ is the Laplacian operator on $M$ and $\nu = \bprodesc{\xi}{N}$ is the so-called \emph{angle function}. Considering the constant function $1$ as a test function, we obtain that
\[
\begin{split}
  \mathcal{Q}(1) &= -\int_M\lvert\sigma\rvert^2 \,\mathrm{d}v-\int_M2(n+1-\tau^2)\, \mathrm{d}v+\int_M2(n+1)(1-\tau^2)\nu^2\, \mathrm{d}v\\
                 &\leq  -\int_M2(n+1-\tau^2)\,\mathrm{d}v+\int_M2(n+1)(1-\tau^2)\, \mathrm{d}v=-\int_M2n\tau^2\, \mathrm{d}v,
\end{split}
\]
where we have used that $\nu^2\leq 1$.  Hence we get the following result:
\begin{proposition}\label{prop:hypersurface-unstable}
Every  orientable compact minimal hypersurface of $\mathbb{S}^{2n+1}_{\tau}$ is unstable and the first eigenvalue of $\mathcal{L}$ is $\lambda_1\leq -2n\tau^2$.
\end{proposition}

Now we define an important family of minimal submanifolds of $\mathbb{S}^{2n+1}_{\tau}$.

\begin{example}\label{ex:induced-bundle} 
  \begin{enumerate}[(1)]
  \item \emph{Induced $\mathbb{S}^1$-bundle by the Hopf fibration over a complex submanifold of $\mathbb{CP}^n(4)$}. Let $\Psi: N^{2m} \to \mathbb{CP}^n(4)$ be  a complex immersion of a  complex manifold $N$ and  
  $$
    M_0^{2m+1} = \{(p,x)\in N\times \mathbb{S}^{2n+1}_{\tau}\colon \Psi(p)=\pi(x)\},
  $$
where $\pi:\mathbb{S}^{2n+1}_{\tau}\rightarrow\mathbb{CP}^n(4)$ is the Hopf fibration. We define 
   $\pi_0: M_0 \to N$  by ${\pi_0}(p,x) = p$ and  $\Phi_0: M_0 \to \mathbb{S}^{2n+1}_\tau$  by $\Phi_0(p,x) = x$. Then, the following diagram is commutative:
\begin{equation}\label{eq:diagram}
		\begin{CD}
			M_0^{2m+1}   @>\Phi_0>>   \mathbb{S}^{2n+1}_\tau \\
			@V{\pi_0}VV                                    @VV{\pi}V\\
			N^{2m}                  @>>\Psi>      \mathbb{CP}^n(4)
		\end{CD}\quad\quad \pi\circ\Phi_0=\Psi\circ{\pi_0},
	\end{equation}
and $M_0^{2m+1}$ is a $\mathbb{S}^1$-bundle over $N^{2n}$, where the action $\mathbb{S}^1\times M_0\rightarrow M_0$ is given by $z \cdot (p,x)=(p,zx)$ for any $z\in\mathbb{S}^1$ and $(p,x)\in M_0$. Now, as  $\xi$ is tangent to $M_0$, i.e. $\xi^{\perp}=0$, and $\Psi$ is a minimal immersion,  it is easy to check  that $\Phi_0$ is  also a minimal immersion, which will be called the induced bundle by the Hopf fibration over the complex immersion $\Psi:N^{2m}\rightarrow\mathbb{CP}^n(4)$.

We observe that $\Phi_0$ is an embedding if and only if $\Psi$ is an embedding. 

  \item \emph{$\mathbb{S}^1$-bundle compatible with the Hopf fibration over a complex submanifold of $\mathbb{CP}^n(4)$.} An immersion $\Phi: M^{2m+1} \to \mathbb{S}^{2n+1}_\tau$ is called a $\mathbb{S}^1$-bundle compatible with the Hopf fibration if there exists a complex immersion $\Psi: N^{2n} \to \mathbb{CP}^n(4)$ such that $M$ is a $\mathbb{S}^1$-bundle over $N$ whose associated projection  $\hat{\pi}:M\rightarrow N$ satisfies $\Psi \circ \hat{\pi} = \pi \circ \Phi$.

    Since in this case $\xi$ is also tangent to $M$ and $\Psi$ is minimal we get that $\Phi$ is a minimal immersion.

    Moreover, if $\Phi_0:M_0^{2m+1}\rightarrow\mathbb{S}^{2n+1}_{\tau}$ is the induced $\mathbb{S}^1$-bundle over $N$, then there exists an integer $s\geq 1$ and a $s$-sheeted covering map $\widetilde{\pi}:M\rightarrow M_0$ such that $\pi_0\circ\widetilde{\pi}=\hat{\pi}$ and $\Phi=\Phi_0\circ\widetilde{\pi}$, where $\pi_0:M_0\rightarrow N$ is the projection. In such case, the following diagrams are commutative
	\begin{equation}\label{eq:diagram-Hopf-compatible-S1-bundle-induced-bundle}
    \xymatrix{
      M^{2m+1} \ar[r]^{\tilde{\pi}} \ar@/^2pc/[rr]^\Phi \ar[dr]_{\hat{\pi}} & M_0^{2m+1} \ar[r]^{\Phi_0} \ar[d]^{\pi_0} & \mathbb{S}^{2n+1}_\tau \ar[d]^\pi \\
                                                                              & N^{2m} \ar[r]_\Psi& \mathbb{CP}^n(4)
    }
	\end{equation}

  In fact, let $\varphi:\mathbb{R} \times M\rightarrow M$ be the uniparametric group of $\hat{\xi}$, where $\hat{\xi}$ is the restriction of $\xi$ to $M$. Then $(\Phi\circ\varphi)(t,p)=\zeta(t,\Phi(p))$ (see \eqref{eq:uniparametric}). If $\lambda\in \mathbb{R}$ is the minimum period of all the curves $\varphi_p$ (see \cite{BW1958}), then $\zeta(t+\lambda,\Phi(p))=\zeta(t,\Phi(p))$ for any $p\in M$ and so there exists an integer $s\geq 1$ such that $\lambda=2\pi\tau s$.

  The action $\mathbb{S}^1 \times M \to M$ is given by
\[
  e^{it}\cdot p=\varphi(s \tau t,p),\quad 0\leq t\leq 2\pi, 
\] 
so the subgroup $G_s\subset \mathbb{S}^1$ given by $G_s=\{1, e^{\frac{i2\pi}{ s}},\dots,e^{\frac{i2\pi(s-1)}{ s}}\}$ also acts on $M$ and $\hat{\Phi}:M/G_s\rightarrow \mathbb{S}^{2n+1}_{\tau}$ given by $\hat{\Phi}([p])=\Phi(p)$ is well-defined and it is also a $\mathbb{S}^1$-bundle over $N$ compatible with the Hopf fibration.
 
Let $\Phi_0:M_0^{2m+1}\rightarrow\mathbb{S}^{2n+1}_{\tau}$ be the induced $\mathbb{S}^1$-bundle over the complex submanifold $\Psi:N\rightarrow\mathbb{CP}^n(4)$. Then it is easy to see that
 \begin{eqnarray*}
F: M/G_s\rightarrow M_0\\
 F([p])=(\hat\pi(p),\Phi(p))
\end{eqnarray*}
is a $\mathbb{S}^1$-bundle isomorphism and so the assertion follows.

The integer $s\geq 1$ will be called the \emph{order of the immersion $\Phi$}.
  \end{enumerate}
\end{example}

In the next result we characterize two important families of minimal submanifolds of $\mathbb{S}^{2n+1}_{\tau}$ which will play an important role along the paper.
\begin{proposition}\label{prop:killing-tangente-normal}
Let $\Phi:M^d\rightarrow\mathbb{S}^{2n+1}_{\tau}$  be a minimal immersion of a  $d$-dimensional manifold $M$.
\begin{enumerate}[(i)]
	\item The normal component of the Killing vector field $\xi$ vanishes identically and the normal bundle of $\Phi$ is invariant under the complex structure $J$ of $\mathbb{CP}^{n+1}(4(1-\tau^2))$ if and only if, $d=2m+1$ is odd and $M$ is a $\mathbb{S}^1$-bundle $\hat{\pi}: M^{2m+1}\rightarrow N^{2m}$ over a complex submanifold $\Psi:N^{2m}\rightarrow \mathbb{CP}^n(4)$ compatible with the Hopf fibration.

  \item  The tangent component of the Killing vector field $\xi$ vanishes identically if and only if, $J(TM)\subset T^{\perp}M$. In particular $\Phi:M^d\rightarrow\mathbb{S}^{2n+1}_{\tau}\subset\mathbb{CP}^{n+1}(4(1-\tau^2))$ is a totally real immersion and $d\leq n$. In this case the metrics induced on $M$ by the Berger metric $\langle{\cdot},{\cdot}\rangle$ and the standard metric $g$ are the same.
\end{enumerate}	
In both cases, the immersion $\Phi:M^d\rightarrow \mathbb{S}^{2n+1}_{\tau'}$ is also minimal for all $\tau'\in (0,1]$.	
\end{proposition}

\begin{proof}
  (i) Suppose that $\xi^\perp = 0$ and the normal bundle of $\Phi$ is invariant under the complex structure $J$. Then, as the normal bundle is of even dimension, we get that $d=2m+1$ for some integer $m\geq 0$. Also, as $\xi^{\perp}=0$, the restriction $\hat{\xi}$ of $\xi$ to $M$ is tangent to $M$ and so $TM$ orthogonally decomposes as $TM=\mathcal{D}\oplus\langle\hat{\xi}\rangle$, where the subbundle $\mathcal{D}$ is invariant under $J$.  

  If $\eta$ is the  $1$-form on $M$ given by $\eta(v)=\langle\hat{\xi},v\rangle$, then using \eqref{eq:Levi-Civita-connection-Berger-round} and that $\xi^{\perp}=0$, it follows that the differential of $\eta$ is given by $\mathrm{d}\eta(X,Y)=2\tau\langle JX,Y\rangle$ for any vector fields $X,Y$ tangent to $M$. So, if $\{e_1,\dots,e_m,Je_1,\dots,Je_m\}$ is an orthonormal basis of   $\mathcal{D}_p, \, p\in M$, we have that
\[
  (\eta\wedge(\mathrm{d}\eta)^m)_p(\hat{\xi}_p,e_1,\dots,e_m,Je_1,\dots,Je_m)\not=0, 
\] 
and hence $\eta$ defines a contact structure on $M$. We remark that, in particular, $M$ is orientable. Now, from a result of W. Boothby and H. Wang~\cite{BW1958} and A. Morimoto~\cite{Mor1964}, we have that $M$ is a $\mathbb{S}^1$-principal bundle over a complex manifold $N^{2m}$. If $\hat{\pi}:M\rightarrow N$ is the projection, then we define $\Psi:N\rightarrow \mathbb{CP}^n(4)$ by $\Psi(q)=\pi(\Phi(p))$, where $p\in\hat{\pi}^{-1}(q)$. It is clear that $\Psi$ is well defined, i.e., it is independent of the point $p$ in $\hat{\pi}^{-1}(q)$, and so $\Psi$ defines a complex immersion of the manifold $N$ in $\mathbb{CP}^n(4)$ satisfying $\Psi \circ \hat{\pi} = \pi \circ \Phi$. More precisely, $\Phi: M \to \mathbb{S}^{2n+1}_\tau$ is a $\mathbb{S}^1$-bundle compatible with the Hopf fibration (see \autoref{ex:induced-bundle}.(2)).

Conversely, if $M^{2m+1}$  is a $\mathbb{S}^1$-bundle $\hat{\pi}: M^{2m+1}\rightarrow N^{2m}$ over a complex submanifold $\Psi:M^{2m}\rightarrow \mathbb{CP}^n(4)$ compatible with the Hopf fibration, then $\xi^\perp = 0$ and, as the normal bundle of $\Psi$ is complex,  then the normal bundle of $\Phi$ is invariant under the complex structure $J$.

(ii) If $\xi^\top=0$, then $\xi\in\mathfrak{X}^{\perp}(M)$, and by \eqref{eq:covariant-derivative-xi} and 	the Weingarten equation:
\[
\tau JX=\nablab_X\xi=-A_{\xi}X +\nabla^{\perp}_X\xi,\quad \forall X\in\mathfrak{X}(M).
\]
So $\langle\sigma(X,Y),\xi\rangle=-\tau\langle JX,Y\rangle$. As $\sigma$ is symmetric and $\langle J-,-\rangle$ skew-symmetric, we obtain that $A_{\xi}=0$ and hence last equation says that $JX\in\mathfrak{X}^\perp(M)$. 

Conversely, if $J(TM) \subset T^\perp M$ then, since $J\xi$ is normal to $\mathbb{S}^{2n+1}_\tau$ in $\mathbb{CP}^{n+1}(4(1-\tau^2))$ (see \autoref{prop:geodesic-spheres-complex-spaces}),
\[
  0 = \bprodesc{JX}{J\xi} = \bprodesc{X}{\xi}, \qquad \text{for any } X \in \mathfrak{X}(M),
\] 
so $\xi$ is normal to $\Phi$ which finishes the proof of (ii).

Finally, in both cases, from \eqref{eq:defini-metrica Berger} it follows that
\[
  0=\bprodesc{v}{\eta}=\prodesc{v}{\eta} = \langle v, \eta \rangle_{\tau'},\quad \forall v\in TM,\,\eta\in T^{\perp}M,\,\tau'\in (0,1],
\]
which means that $T^{\perp}M=T^{\perp}_{\tau'}M$, where $T^\perp_{\tau'}M$ stands for the normal bundle of the immersion $\Phi: M \to \mathbb{S}^{2n+1}_{\tau'}$, for any $\tau'\in (0,1]$. Under this condition, taking normal components in \eqref{eq:Levi-Civita-connection-Berger-round}, the mean curvature $H_g$ of the immersion $\Phi$ with respect to the metric $g$ is given by 
\[
dH_g=\tfrac{2(1-\tau^2)}{\tau}\sum_{i=1}^d\langle \xi,e_i\rangle (Je_i)^\perp=(J\xi^{\top})^{\perp}=0,
\]
which means that $\Phi$ is also minimal with respect to the metric $g$. Repeating the argument, we get that $\Phi:M\rightarrow \mathbb{S}^{2n+1}_{\tau'}$ is minimal for any $\tau'\in (0,1]$. 
\end{proof}

\begin{proposition}\label{prop:induced-index}
  Let $\Phi: M^{2m+1} \to \mathbb{S}^{2n+1}_\tau$ be a $\mathbb{S}^1$-bundle compatible with the Hopf fibration over a complex submanifold $\Psi: N \to \mathbb{CP}^n(4)$. Then, the multiplicities of the eigenvalues of the Jacobi operator of $\Phi$ are even. In particular, its index and nullity are even. 
\end{proposition}

\begin{proof}
Firstly, as $\xi^{\perp}=0$, we can decompose $TM=\mathcal{D}\oplus\langle\xi\rangle$. As $\Psi$ is a complex immersion, $\mathcal{D}$ and the normal bundle $T^{\perp}M$ are invariant by the complex structure $J$ of $\mathbb{CP}^{n+1}(4(1-\tau^2))$. Hence, by \eqref{eq:Jacobi-operator}, the Jacobi operator is 
\begin{equation}\label{eq:Jacobi-operator-family-xi-tangent}
\mathcal{L} \eta = \Delta^\perp \eta + \mathcal{A} \eta + (2m + \tau^2)\eta.
\end{equation}

Moreover, as $TM=\mathcal{D}\oplus\langle\xi\rangle$, then for any  vector field $X \in\Gamma(\mathcal{D})$,  \eqref{eq:covariant-derivative-xi} and the fact that $\mathcal{D}$ is invariant by $J$ imply that 
\begin{equation}\label{eq:properties-covariant-derivative-2ff-immersion-xi-tangent-J-normal}
\nabla_{\xi}\xi=0,\quad	\nabla_X \xi = \tau JX, \quad\sigma(\xi,\xi)=0,\quad \sigma(X, \xi) = 0.
\end{equation}
If $\nablap$ is the Levi-Civita connection of $\mathbb{CP}^{n+1}(4(1-\tau^2))$, using the Gauss and Weingarten formulae joint with  \eqref{eq:2ff-Berger-complex-space} and \eqref{eq:properties-covariant-derivative-2ff-immersion-xi-tangent-J-normal} in the equation $\nablap J=0$, it is easy to get that 
\begin{equation}\label{eq:properties-immersion-xi-tangent-J-normal}
	\begin{gathered}
		\nabla_X JY = J\bigl(\nabla_X Y +\tau \bprodesc{JX }{Y}\xi\bigr) - \tau \bprodesc{X}{Y}\xi, \quad\quad \sigma(X, JY) = J \sigma(X, Y)\\
	A_{J\eta}X = J A_\eta X=-A_{\eta}JX,\quad\quad
		\nabla^\perp_X J\eta = J \nabla^\perp_X \eta,
	\end{gathered}
\end{equation} 
for any $X,Y \in\Gamma( \mathcal{D})$   and any $\eta \in \mathfrak{X}^\perp(M)$.

Using \eqref{eq:properties-immersion-xi-tangent-J-normal} we easily get that $\Delta^{\perp}J=J\Delta^{\perp}$ and $\mathcal{A}J=J\mathcal{A}$, and so $\mathcal{L}J=J\mathcal{L}$. This implies that the eigenspace associated to an eigenvalue of $\mathcal{L}$ is invariant under $J$ and hence of dimension even.  
\end{proof}

Now, we are going to give explicit examples of the two families of minimal submanifolds of $\mathbb{S}^{2n+1}_{\tau}$  studied in \autoref{prop:killing-tangente-normal}.

\begin{example}\label{ex:examples-induced-bundle-by-Hopf-fibration}
  Using the easiest examples of complex submanifolds of $\mathbb{CP}^n(4)$ and following \autoref{ex:induced-bundle},  we give two explicit nice examples of minimal submanifolds of $\mathbb{S}^{2n+1}_{\tau}$ of the family (i) of \autoref{prop:killing-tangente-normal}:

  \begin{enumerate}[(1)]
    \item \label{ex:sphere} Let $m \geq 0$ and $\Psi: \mathbb{CP}^m(4) \to \mathbb{CP}^n(4)$ be the totally geodesic complex embedding given by $\Psi([(z_1,\ldots,z_{m+1})]) = [(z_1,\ldots,z_{m+1},0,\ldots,0)]$. Then, the induced bundle by the Hopf fibration is the minimal embedding $\Phi_0: \mathbb{S}^{2m+1}_\tau \to \mathbb{S}^{2n+1}_\tau$ given by $\Phi(z_1,\ldots, z_{m+1}) = (z_1,\ldots,z_{m+1},0,\ldots,0)$. 

    If $m \geq 1$ then $\mathbb{S}^{2m+1}$ does not have any covering. If $m = 0$ then, for any $s \in \mathbb{N}$ there exists a $s$-sheeted covering $\tilde{\pi}: \mathbb{S}^1 \to \mathbb{S}^1$, $\tilde{\pi}(z) = z^s$, and the corresponding totally geodesic immersion $\Phi_s=\Phi_0\circ\tilde{\pi}:\mathbb{S}^1\rightarrow\mathbb{S}^{2n+1}_{\tau}$ is given by $\Phi_s(z) = (z^s,0,\ldots,0)$. Moreover, the induced metric by $\Phi_s$ in $\mathbb{S}^1$ is $s^2 \tau^2 g$.

    \item \label{ex:Veronese}
  Let $\Psi: \mathbb{C}\mathbb{P}^1(2) \to \mathbb{C}\mathbb{P}^2(4)$ be the degree two Veronese complex embedding given by $\Psi([z,w]) = [(z^2, \sqrt{2}zw, w^2)]$ and $$ M_0=\{\bigl([(z,w)], x\bigr)\in \mathbb{CP}^1(2)\times\mathbb{S}^5_{\tau}\,:\, [z,w]\in\mathbb{CP}^1(2),\, \pi(x)=[z^2,\sqrt{2}zw,w^2]\} $$  its induced bundle by the Hopf fibration. Then $M_0$ can be identified with the three dimensional real projective space $\mathbb{R}\mathbb{P}^3 = \{ [\![(z,w)]\!] \in \mathbb{S}^3 / \{\mathrm{I}, -\mathrm{I}\}\}$ via the diffeomorphism $\mathbb{RP}^3 \to M_0$ given by
  \[
    [\![(z,w)]\!] \mapsto \bigl([(z,w)], (z^2, \sqrt{2}zw, w^2)\bigr).
  \] 
  Now, the 	corresponding minimal embedding $\Phi_0: \mathbb{RP}^3 \to \mathbb{S}^5_\tau$ is given by $\Phi_0([\![(z,w)]\!]) = (z^2, \sqrt{2}zw, w^2)$.

  In this case, $\tilde{\pi}:\mathbb{S}^3\rightarrow\mathbb{RP}^3$ is a $2$-sheeted covering and the minimal immersion $\Phi=\Phi_0\circ\tilde{\pi}:\mathbb{S}^3\rightarrow\mathbb{S}^5_{\tau}$ (see \autoref{ex:induced-bundle}.(2)) is given by $\Phi(z, w) = (z^2, \sqrt{2}zw, w^2)$. The map $\pi_0\circ\tilde{\pi}:\mathbb{S}^3\rightarrow\mathbb{CP}^1(2)$ is the Hopf fibration and the induced metric on $\mathbb{S}^3$ by the immersion $\Phi$ is $2\langle \cdot, \cdot\rangle_{\sqrt{2}\tau}$.
  \end{enumerate}
\end{example}

In the next result we classify the totally geodesic submanifolds of $\mathbb{S}^{2n+1}_{\tau}$, appearing an example of the family described in \autoref{prop:killing-tangente-normal}.(ii).  Although the result must be known, we have not find it in the literature and  we include it for completeness.

\begin{proposition}\label{prop:classification-totally-geodesic}
	Let $\Phi:M^d\rightarrow\mathbb{S}^{2n+1}_{\tau}, 0<\tau<1$,  be an immersion of a   $d$-dimensional manifold $M$. Then $\Phi$ is totally geodesic if and only if, up to congruences, $\Phi(M)$ is an open subset of either one of the following spheres of $\mathbb{S}^{2n+1}_{\tau}$:
	\begin{enumerate}[(i)]
    \item A Berger sphere $\mathbb{S}^{2m+1}_\tau$, $m \geq 0$ given in \autoref{ex:examples-induced-bundle-by-Hopf-fibration}.(\ref{ex:sphere})
		\item $\mathbb{S}^d = \{(a_1, \ldots, a_{d+1}, 0, \ldots,0) \in \mathbb{S}^{2n+1}_\tau \subset \mathbb{C}^{n+1} \colon a_i \in \mathbb{R}\}, 1 \leq d\leq n.$ 
	\end{enumerate}
  or $\Phi(M)$ is a geodesic.
\end{proposition}

\begin{remark}
  The Killing field $\xi$ satisfies  \ $\xi^\top = 0$ in the example     (ii). The cases $m = 0$ in (i) and $d = 1$ in (ii) correspond to embedded closed geodesic in $\mathbb{S}^{2n+1}_\tau$.
\end{remark}

\begin{proof}
	Firstly, the submanifold presented in item (i)  is totally geodesic since it is the fixed point set of the isometry $A = \left( \begin{smallmatrix} \mathrm{Id}_{m+1} & 0 \\ 0 &  -\mathrm{Id}_{n-m}\end{smallmatrix} \right) \in U(n+1)$.

Also, as item (ii) is also totally geodesic with respect to the metric $g$ on $\mathbb{S}^{2n+1}$, then $\nablar_X Y$ is tangent to $\mathbb{S}^d$ for any $X,Y\in\mathfrak{X}(S^d)$.  Moreover, $\xi^{\top}=0$ and so, from ~\eqref{eq:Levi-Civita-connection-Berger-round}, it follows that $\sigma(X,Y)=0$, and hence $\mathbb{S}^d$ is totally geodesic in $\mathbb{S}^{2n+1}_{\tau}$.

	Suppose now that  $\Phi: M^d \to \mathbb{S}^{2n+1}_\tau$ is a totally geodesic immersion with $d\geq 2$. We consider  the orthogonal decomposition $\xi = \xi^\top + \xi^\perp$ of $\xi$ in tangential and normal components to $\Phi$. Now, by \eqref{eq:covariant-derivative-xi} and as $\sigma=0$, 
	\begin{equation}\label{eq:derivatives-T-N-totally-geodesic}
		\tau J(X - \bprodesc{X}{\xi}\xi) = \nablab_X \xi = \nablab_X(\xi^\top + \xi^\perp) = \nabla_X \xi^\top + \nabla^\perp_X \xi^\perp,
	\end{equation}
	which implies that $\nabla_X \xi^\top = \tau(J X)^\top$ and so $\xi^\top$ is a Killing field over $M$. But  any Killing vector field on $M$ satisfies $\nabla_X \nabla_X \xi^\top  - \nabla_{\nabla_XX} \xi^\top + R(\xi^\top,X)X = 0$ for any tangent vector field $X$.  Now we compute the members of this equation  for any $X$ orthogonal to $\xi^\top$ with $\lvert X\rvert^2=1$. As $\sigma=0$, from Gauss equation and \eqref{eq:Riemann-tensor-Berger} we obtain
\[
R(\xi^\top,X)X=(\bR(\xi^\top,X)X)^{\top}=\big(1-(1-\tau^2)\lvert \xi^\top\rvert^2\big)\xi^\top+3(1-\tau^2)\langle\xi^\top,JX\rangle(JX)^\top.
\]
Also, if $\nablap$ is the Levi-Civita connection of $\mathbb{CP}^{n+1}(4(1-\tau^2))$, as $(\nablap_XJ)X=0$, taking into account that $\sigma=0$ and equations \eqref{eq:2ff-Berger-complex-space} and \eqref{eq:derivatives-T-N-totally-geodesic} we get that
\[
\nabla_X \nabla_X \xi^\top  - \nabla_{\nabla_XX} \xi^\top=\tau\{\nabla_X(JX)^\top-(J\nabla_XX)^\top\}=-\tau\hat{\sigma}(X,X)\xi^\top=-\tau^2\xi^{\top}.
\]
From last two equations we finally obtain that 
	\[
	(1-\tau^2)\bigl[ (1-\lvert\xi^\top\rvert^2)\xi^\top + 3 \bprodesc{\xi^\top}{J X}(J X)^\top\bigr] = 0.
	\] 
	Multiplying by $\xi^\top$ and as  $\tau \not= 1$ we obtain that either $\lvert\xi^\top\rvert = 1$ (i.e.\ $\xi^{\perp}=0$), or $\xi^\top= 0$.

Now, the normal bundles of $\Phi$ with respect to the Berger metric and the standard one $g$ are the same. Hence, from \eqref{eq:Levi-Civita-connection-Berger-round} and taking into account \eqref{eq:derivatives-T-N-totally-geodesic} we deduce that, the second fundamental form $\sigma_g$ of $\Phi: M \to \mathbb{S}^{2n+1}$ with respect to the standard metric satisfies:
\[
\sigma_g(X,Y)=\tfrac{1-\tau^2}{\tau^2}(\langle Y,\xi^\top\rangle\nabla^\perp_X\xi^\perp+\langle X,\xi^\top\rangle\nabla^\perp_Y\xi^\perp)=0.
\]
Hence $\Phi: M \to \mathbb{S}^{2n+1}$ is totally geodesic, and so $\Phi(M)$ is an open subset of a $d$-dimensional sphere  $\mathbb{S}^d\subset\mathbb{S}^{2n+1}$. 

In the case $\xi^\perp = 0$, the complex structure $J$ leaves invariant the normal bundle (see \eqref{eq:covariant-derivative-xi}) so, using \autoref{prop:killing-tangente-normal}, we obtain that, up to an isometry of the Berger sphere, $\Phi$ is an embedding and $\Phi(M)$ is an open subset of one  of the examples shown in the result. When $\xi^{\top}=0$, the result follows from \autoref{prop:killing-tangente-normal}.(ii).
\end{proof}

\begin{example}[Clifford hypersurfaces]\label{ex:Clifford-hypersurfaces}
Now,  we are going to introduce  a family of minimal hypersurfaces of $\mathbb{S}^{2n+1}_{\tau}$, that we will name \emph{Clifford hypersurfaces}. These examples will be those classical Clifford hypersurfaces of $\mathbb{S}^{2n+1}$ which satisfies $\xi^{\perp}=0$ ($\nu=0$), which, from \autoref{prop:killing-tangente-normal}, will be minimal in $\mathbb{S}^{2n+1}_{\tau}$ for any $\tau\in(0,1]$.

Let consider, for any integers $d_1, d_2 \geq 1$ with $d_1 + d_2 = 2n$, the canonical embedding $\Phi: \mathbb{S}^{d_1}(r_1) \times  \mathbb{S}^{d_2}(r_2) \to \mathbb{S}^{2n+1}$ where $r_1=\sqrt{\frac{d_1}{2n}}$ and $r_2=\sqrt{\frac{d_2}{2n}}$ are the radii of the spheres. It is well-known that $\Phi$ is  minimal, that $N_{(p,q)}=(\sqrt{\frac{d_2}{d_1}}p,-\sqrt{\frac{d_1}{d_2}}q)$ is a unit normal vector field to $\Phi$ and that $\lvert\sigma^g\rvert_g^2=d_1+d_2$. Under these conditions it is not difficult to check that $g(N_{(p,q)},i(p,q))=0$ for any $(p,q)\in \mathbb{S}^{d_1}(r_1)\times\mathbb{S}^{d_2}(r_2)$ if and only if $d_1=2m_1+1,\, d_2=2m_2+1$ and, up to congruences,
\[
\mathbb{S}^{2m_1+1}(r_1)\times \mathbb{S}^{2m_2+1}(r_2)=\{(z,w)\in\mathbb{C}^{m_1+1}\times\mathbb{C}^{m_2+1}\colon \lvert z\rvert=r_1,\,\lvert w\rvert=r_2\}\subset\mathbb{S}^{2n+1}_{\tau}.
\]
In fact, since $i(p,q)$ is a unit tangent vector field to $\mathbb{S}^{d_1}(r_1) \times  \mathbb{S}^{d_2}(r_2)$, the Poincaré-Hopf theorem ensures that the Euler characteristic of $\mathbb{S}^{d_1}(r_1) \times  \mathbb{S}^{d_2}(r_2)$ is $0$. This only happens when $d_1$ and $d_2$ are odd.

Now, \autoref{prop:killing-tangente-normal} says that the embedding $\Phi:\mathbb{S}^{2m_1+1}(r_1)\times \mathbb{S}^{2m_2+1}(r_2)\rightarrow\mathbb{S}^{2n+1}_{\tau}$ is a minimal hypersurface for any $\tau\in(0,1]$, satisfying $\xi^{\perp}=0$. From \eqref{eq:defini-metrica Berger}  the normal bundles with respect to the Euclidean metric and the Berger metric are the same and both metric coincide on the normal bundle. So $N$ is also a unit normal vector field to $\Phi:\mathbb{S}^{2m_1+1}(r_1)\times \mathbb{S}^{2m_2+1}(r_2)\rightarrow\mathbb{S}^{2n+1}_{\tau}$ and from \eqref{eq:Levi-Civita-connection-Berger-round} it follows that 
\[
  \prodesc{\sigma^g(X,Y)}{N}=\bprodesc{\sigma(X,Y)}{N} + \tfrac{1-\tau^2}{\tau}[ \bprodesc{Y}{\xi} \bprodesc{JX}{N} + \bprodesc{X}{\xi} \bprodesc{JY}{N}].
\]
From last formula it follows, taking into account that $\lvert\sigma^g\rvert^2_g=2(m_1+m_2+1)$, that 
$\lvert\sigma\rvert^2 = 2(m_1+m_2+\tau^2)$.
\end{example}

In the following results we study the index and the nullity of the  Examples~\ref{ex:examples-induced-bundle-by-Hopf-fibration} and~\ref{ex:Clifford-hypersurfaces}.

\begin{proposition}\label{prop:index-nullity-totally-geodesic-Berger-spheres}
  Let $\Phi:\mathbb{S}^{2m+1}_\tau\rightarrow\mathbb{S}^{2n+1}_\tau$ be the totally geodesic embedding given in \autoref{ex:examples-induced-bundle-by-Hopf-fibration}.(\ref{ex:sphere}). Then:
	\begin{enumerate}[(i)]
		\item $\Ind(\mathbb{S}^{2m+1}_\tau) = \begin{cases}
			0, & \text{if } \tau^2 \leq \frac{1}{2(m+1)}, \\
			2(n-m), & \text{if } \frac{1}{2(m+1)} < \tau^2 \leq 1, \\
		\end{cases}$
		
		\item $\Nul(\mathbb{S}^{2m+1}_\tau) = \begin{cases}
      2(n-m)(m+1), & \text{if } \tau^2 < 1 \text{ and } \tau^2 \neq \frac{1}{2(m+1)}, \\
			2(n-m)(m+2), & \text{if } \tau^2 = \frac{1}{2(m+1)}, \\
      4(n-m)(m+1), & \text{if } \tau^2 = 1.
		\end{cases}$
	\end{enumerate}

    Moreover, let $m = 0$ and $\Phi_s: \mathbb{S}^1 \to \mathbb{S}^{2n+1}_\tau$ be the totally geodesic immersion $\Phi_s(z) = (z^s,0,\ldots,0)$ described in \autoref{ex:examples-induced-bundle-by-Hopf-fibration}.(\ref{ex:sphere}). Then $\Phi_s$ is stable if and only if $\tau^2 \leq \frac{1}{2s}$.
\end{proposition}

\begin{proof} 
From \eqref{eq:Jacobi-operator-family-xi-tangent} and as $\sigma=0$,  we have that the Jacobi operator of $\Phi$ is given by 
  \[
  \mathcal{L}\eta=\Delta^{\perp}\eta+(2m+\tau^2)\eta,
\] 
for any $\eta\in\mathfrak{X}^\perp(\mathbb{S}^{2m+1}_{\tau})$. First, considering the spheres embedded in $\mathbb{C}^{n+1}$, it is clear that $\{a_j,ia_j\colon 1\leq j\leq n-m\}$ defined by 
\[
a_j=(0,\dots,0,\stackrel{m+1+j}{1},0,\dots,0)\in\mathbb{C}^{n+1}
\]
is a global orthonormal reference of the normal bundle of $\Phi$. If $\mathcal{D}_j = \Span\,\{ a_j, ia_j\}$, then the normal bundle can be decomposed as $T^\perp \mathbb{S}^{2m+1}_\tau = \mathcal{D}_1 \oplus \cdots \oplus \mathcal{D}_{n-m}$. Hence any normal section $\eta\in\mathfrak{X}^\perp(\mathbb{S}^{2m+1}_{\tau})$ can be written as $\eta=\sum_{j=1}^{n-m}\{f_ja_j+g_j\, ia_j\}$, with $f_j,g_j\in \mathcal{C}^{\infty}(\mathbb{S}^{2m+1}_{\tau})$.

 Since $\nablar_u a_j =\nablar_u ia_j =0$ for any tangent vector $u\in T\mathbb{S}^{2m+1}_{\tau}$,  we easily get from \eqref{eq:Levi-Civita-connection-Berger-round} and \eqref{eq:properties-immersion-xi-tangent-J-normal} that
 \[
   \nabla^\perp_u a_j = - \tfrac{1-\tau^2}{\tau} \bprodesc{u}{\xi}Ja_j,\quad \Delta^\perp a_j = - \tfrac{(1-\tau^2)^2}{\tau^2}a_j, \quad \nabla^\perp_u i a_j = i \nabla^\perp_u a_j, \quad \Delta^\perp ia_j = i \Delta^\perp a_j.
 \] 
 
 Hence, the Jacobi operator  is given by 
\begin{multline*}
    \mathcal{L} \eta = 
                                 \sum_{j=1}^{n-m}\left[ \Delta f_j + (2m + 2 - \tfrac{1}{\tau^2})f_j + 2 \tfrac{1-\tau^2}{\tau} L_\xi g_j\right] a_j \\+ 
                                \sum_{j=1}^{n-m}\left[ \Delta g_j + (2m + 2 - \tfrac{1}{\tau^2})g_j - 2 \tfrac{1-\tau^2}{\tau} L_\xi f_j\right]  ia_j,\quad \forall \,\eta=\sum_{j=1}^{n-m}\{f_ja_j+g_j \, ia_j\}
.
\end{multline*} 
As a consequence, if $\Gamma(\mathcal{D}_j)$ is the space of sections of the subbundle $\mathcal{D}_j$, then $\mathcal{L}(\Gamma(\mathcal{D}_j))\subset\Gamma(\mathcal{D}_j)$, and  so $\Ind(\mathcal{L}) = \sum_{j=1}^{n-m} \Ind(\mathcal{L}_j)$, where $\mathcal{L}_j$ is the restriction of $\mathcal{L}$ to $\Gamma(\mathcal{D}_j)$.

 We are going to compute the  index and nullity of $\mathcal{L}_j$ and prove  that these numbers are independent of the subbundles $\mathcal{D}_j$.  To simplify the notation, let $\mathcal{D}=\Span\{a,ia\}$ and $\mathcal{L}:\Gamma(\mathcal{D})\rightarrow\Gamma(\mathcal{D})$.

Let $\eta = f a + g (ia)$ be an eigensection of $\mathcal{L}$ associated to the eigenvalue $\rho$, i.e.\ $\mathcal{L} \eta + \rho \eta = 0$. Then, from the previous equation we get
\begin{equation}\label{eq:totally-geodesic-Berger-sphere-decomposition-Jacobi-operator-fa+gJa}
\begin{split}
  \Delta f + [\rho + (2m+2 - \tfrac{1}{\tau^2})]f &= -2 \tfrac{1-\tau^2}{\tau} L_\xi g,  \\
  \Delta g + [\rho + (2m+2 - \tfrac{1}{\tau^2})]g &= 2 \tfrac{1-\tau^2}{\tau} L_\xi f.
\end{split}
\end{equation}
From \autoref{lm:connection-Berger}.(4) we have that $\Delta\circ L_{\xi}=L_{\xi}\circ\Delta$, and so applying $L_{\xi}$ to both equations of ~\eqref{eq:totally-geodesic-Berger-sphere-decomposition-Jacobi-operator-fa+gJa} we obtain
\begin{equation}\label{eq:totally-geodesic-Berger-sphere-decomposition-Jacobi-operator-fa+gJa1}
\begin{split}
  \Delta L_{\xi}f + [\rho + (2m+2 - \tfrac{1}{\tau^2})]L_{\xi}f &= -2 \tfrac{1-\tau^2}{\tau} (L_\xi)^2 g,  \\
  \Delta L_{\xi}g + [\rho + (2m+2 - \tfrac{1}{\tau^2})]L_{\xi}g &= 2 \tfrac{1-\tau^2}{\tau} (L_\xi)^2 f.
\end{split}
\end{equation}

Let $f = \sum_{k,p} f_{k,p}$ and $g = \sum_{k} g_{k,p}$ be the decompositions of $f$ and $g$ in eigenfunctions of the Laplacian $\Delta$ of $\mathbb{S}^{2m+1}_\tau$ with $f_{k,p}, g_{k,p} \in V(\mu_{k,p})$ (see \autoref{lm:connection-Berger}). More precisely,
\[
  \Delta f_{k,p} +\mu_{k,p}f_{k,p}=0, \quad \Delta g_{k,p} +\mu_{k,p}g_{k,p}=0,\quad\mu_{k,p} = k(2m+k) + \tfrac{1-\tau^2}{\tau^2}(k-2p)^2,
\]
with $k$ and $p$ integer numbers satisfying $k \geq 0$ and $0 \leq p \leq \lfloor \tfrac{k}{2} \rfloor$.

As $(L_\xi)^2 f_{k,p} = -\tfrac{1}{\tau^2}(k -2p)^2 f_{k,p}$, 
from \eqref{eq:totally-geodesic-Berger-sphere-decomposition-Jacobi-operator-fa+gJa1}, it follows that 
\begin{equation}\label{eq:totally-geodesic-Berger-sphere-relation-eigenvalues}
\begin{split}
  [\bigl(-\mu_{k,p} + \rho + (2m+2 - \tfrac{1}{\tau^2})\bigr)^2 - 4 \tfrac{(1-\tau^2)^2}{\tau^4}(k - 2p)^2] f_{k,p} &= 0, \\
  [\bigl(-\mu_{k,p} + \rho + (2m+2 - \tfrac{1}{\tau^2})\bigr)^2 - 4 \tfrac{(1-\tau^2)^2}{\tau^4}(k - 2p)^2] g_{k,p} &= 0, \\
\end{split}
\end{equation}
for any $k \geq 0$ and $0 \leq p \leq \lfloor \tfrac{k}{2} \rfloor$, and so,   the eigenvalue $\rho$ of $\mathcal{L}$ takes the form
\begin{equation}\label{eq:totally-geodesic-Berger-spheres-eigenvalues-Jacobi-operator}
  \rho = (2m+1+k)(k-1) + \tfrac{1-\tau^2}{\tau^2}(k - 2p \pm 1)^2.
\end{equation} 
for some $k \geq 0$ and $0 \leq p \leq \lfloor \tfrac{k}{2} \rfloor$. Notice that $\rho > 0$ for $k \geq 2$ and so we only need to analyse the cases $k = 0$ and $k = 1$. In both cases we know by the decomposition in \autoref{lm:connection-Berger}.(4) that $V(\mu_{0,0}) = V(\lambda_0)$ and $V(\mu_{1,0}) = V(\lambda_1)$. 

On the one hand, if $k=0$, then $p=0$ and $\rho=\frac{1}{\tau^2}-(2m+2)$. Hence, in this case,  $\rho$ is positive if $\tau^2<\frac{1}{2m+2}$, $\rho=0$ if $\tau^2=\frac{1}{2m+2}$ and $\rho$ is negative if $\tau^2>\frac{1}{2m+2}$. Moreover, from \eqref{eq:totally-geodesic-Berger-sphere-relation-eigenvalues} and if $\tau^2 \neq \frac{1}{2m+2}$, we deduce that $f$ and $g$ are constant functions and so the eigensections associated to $\rho$ are of the form $\alpha a + \beta ia$, with $\alpha, \beta \in \mathbb{R}$. From here it follows that $\Ind(\mathcal{L})=0$ for $\tau^2\leq \frac{1}{2m+2}$. Also we obtain that, for $\tau^2 > \tfrac{1}{2m+2}$, $\Ind(\mathcal{L})$ is twice the multiplicity of $\mu_{0,0}$, which is one. So, in this case, $\Ind(\mathcal{L})=2$.

On the other hand, when $k=1$, we have that $p=0$ and so $\rho=0$. Now, by \eqref{eq:totally-geodesic-Berger-sphere-relation-eigenvalues} and assuming $\tau^2 \neq \frac{1}{2m+2}$, we get that $f$ and $g$ are eigenfunctions of the Laplacian associated to the eigenvalue $\mu_{1,0} = 2m+1 + \tfrac{1-\tau^2}{\tau^2}$. But, if $\tau^2 \neq 1$, using \eqref{eq:totally-geodesic-Berger-sphere-decomposition-Jacobi-operator-fa+gJa} we get that $g = -\tau L_\xi f$ and so the eigensections associated to $\rho$ are $f a - \tau (L_\xi f) ia$, with $f \in V(\mu_{1,0}) = V(\lambda_1)$. Finally, if $\tau^2 = \frac{1}{2m+2}$ then $\alpha a + \beta ia$, $\alpha, \beta \in \mathbb{R}$ are also eigensections associated to $\rho = 0$ by the analysis in the previous paragraph. We deduce that, if $\tau^2\not=\frac{1}{2m+2}$ and $\tau^2 \neq 1$, $\Nul(\mathcal{L})$ coincides with the multiplicity of $\mu_{1,0}$  that is, $\Nul(\mathcal{L})=(2m+2)$. When $\tau^2=\frac{1}{2m+2}$, then $\Nul(\mathcal{L})$ is increased by $2$ and so, in this case, $\Nul(\mathcal{L})=2m+4$. If $\tau^2 = 1$ then the normal section $\eta = fa + g ia$ with $f, g \in V(\lambda_1)$ satisfies $\mathcal{L} \eta = 0$ and so the nullity in this case is $2(2m+2)$ (see also \cite[Proposition 5.1.1]{Simons68}).

We finally observe that the index and the nullity of $\mathcal{L}$ is the same for any subbundle $\mathcal{D}_j$, and so the result follows.

  We finally analyse the immersion $\Phi_s: \mathbb{S}^1 \to \mathbb{S}^{2n+1}_\tau$, $\Phi_s(z) = (z^s,\ldots,0)$. Notice that the Laplacian $\Delta$ of the induced metric by $\Phi_s$ is $\Delta = (L_\xi)^2$ and if $f$ is an eigenfunction associated to the $k$-th eigenvalue of the Laplacian then $\Delta f = (L_\xi)^2 f = -\frac{k^2}{s^2\tau^2}f$. Following the previous arguments we easily get that the eigenvalues $\rho$ of $\mathcal{L}$ take the form (cp.~\eqref{eq:totally-geodesic-Berger-spheres-eigenvalues-Jacobi-operator})
  \[
    \rho_\pm(k)= (\tfrac{k^2}{s^2} - 1) + \tfrac{1-\tau^2}{\tau^2}(\tfrac{k}{s} \pm 1)^2.
  \] 
  Now, it is clear that $\rho_\pm(k) \geq 0$ for $k \geq s$ and, if we assume $\tau^2 \leq \frac{1}{2s}$ it is not difficult to show that $\rho_\pm(k)$ is also non-negative for $0 \leq k \leq s-1$. Hence, the immersion $\Phi_s$ is stable for $\tau^2 \leq 1/2s$. Moreover, by a direct computation
  \[
    \rho_-(s-1) = \tfrac{1}{s^2}(\tfrac{1}{\tau^2} - 2s),
  \] 
  with associated eigensection $f a + \tfrac{\tau s}{s-1} (L_\xi f) ia$ if $s \neq 1$ or $a$ and $ia$ if $s = 1$, where $f$ is any eigenfunction of $\Delta$ associated to the $(s-1)$-th eigenvalue. Hence, $\Phi_s$ is unstable if $\tau^2 > \frac{1}{2s}$ and the proof finishes.
\end{proof}

\begin{proposition}\label{prop:index-nullity-Veronese}
  Let $\Phi_0: \mathbb{RP}^3 \to \mathbb{S}^{5}_\tau$ be the minimal embedding given in \autoref{ex:examples-induced-bundle-by-Hopf-fibration}.(\ref{ex:Veronese}). Then:
  \[
    \Ind(\mathbb{RP}^3) = 
  \begin{cases}
    8 & \text{if } \frac{1}{2} < \tau^2 \leq 1, \\
    6 & \text{if } \frac{1}{4} < \tau^2 \leq \frac{1}{2} \\
    0 & \text{if } \tau^2 \leq \frac{1}{4}. \\
  \end{cases},
  \qquad
  \Nul(\mathbb{RP}^3) = 
  \begin{cases}
    16 & \text{if } \tau^2 = 1 \text{ or } \tau^2 = \frac{1}{4}, \\
    12 & \text{if } \tau^2 = \frac{1}{2}, \\
    10 & \text{otherwise}.
  \end{cases}
  \] 

  Moreover, if $\Phi: (\mathbb{S}^3, 2 \bprodesc{\cdot}{\cdot}_{\sqrt{2}\tau}) \to \mathbb{S}^5_{\tau}$ is the minimal isometric immersion given in \autoref{ex:examples-induced-bundle-by-Hopf-fibration}.(\ref{ex:Veronese}), then 
  \[
    \Ind(\mathbb{S}^3) = 
    \begin{cases}
      \geq 14 & \text{if } \frac{1}{4} < \tau^2 \leq 1, \\
      8 & \text{if } \frac{1}{8} < \tau^2 \leq \frac{1}{4} \\
      0 & \text{if } \tau^2 \leq \frac{1}{8}. \\
    \end{cases},
    \quad
    \Nul(\mathbb{S}^3) = 
    \begin{cases}
      16 & \text{if } \tau^2 = 1 \text{ or } \tau^2 = \frac{1}{4}, \\
      18 & \text{if } \tau^2 = \frac{1}{8}, \\
      \geq 10 & \text{otherwise.}
    \end{cases}
  \] 
 
\end{proposition}

\begin{proof}
  Let consider the global orthonormal reference of the normal bundle of $\Phi$ given by $\{N, iN\}$ where $N_{(z,w)} = (\bar{w}^2, -\sqrt{2} \bar{z} \bar{w}, \bar{z}^2)$.  Then, given an arbitrary normal section $\eta = fN + g iN$, $f, g \in \mathcal{C}^\infty(\mathbb{S}^3)$,  straightforward computations shows that the Jacobi operator $\mathcal{L}$ (see \eqref{eq:Jacobi-operator-family-xi-tangent}) satisfies  
  \[
  \mathcal{L} \eta = \left[ \Delta f + 4\left(2 - \tfrac{1}{\tau^2}\right) f + 2 \tfrac{(2-\tau^2)}{\tau} L_\xi g \right] N + \left[ \Delta g + 4\left(2 - \tfrac{1}{\tau^2}\right)g  - 2 \tfrac{(2-\tau^2)}{\tau} L_\xi f \right] iN.
\]
 
Decomposing $f = \sum f_{k,p}$ and $g = \sum g_{k,p}$, where $f_{k,p}$ and $g_{k,p}$ are eigenfunction of the Laplacian associated to the eigenvalue $\mu_{k,p}$ (see \autoref{lm:connection-Berger}.(4)), and following a simular argument as in the proof of \autoref{prop:index-nullity-totally-geodesic-Berger-spheres} we can deduce that the eigenvalues of $\mathcal{L}$ take the form
  \[
    \rho_\pm(k,p) = \frac{1}{2} (1 + k(2+k)) + \frac{1}{4\tau^2}(k - 2p \pm 4)^2 - 8 - \frac{1}{2}(k - 2p \pm 1)^2.
  \] 
  for certain integers $k \geq 0$ and $0 \leq p \leq \lfloor \frac{k}{2}\rfloor$. Moreover, the associated eigensections to $\rho_\pm(k,p)$ are
\[
  f_{k,p} N \pm \tfrac{2\tau}{k - 2p} L_\xi f_{k, p} iN \quad\text{if } k - 2p \neq 0, \quad\text{or}\quad f_{k,p}N \text{ and } f_{k,p}iN  \quad \text{if } k = 2p.
\] 
Hence, the multiplicity associated to $\rho_\pm(k,p)$ is
\[
  \dim V(\mu_{k,p}) \text{ if } k \neq 2p, \quad \text{or}\quad 2\cdot \dim V(\mu_{k,k / 2}) \text{ if } k = 2p.
\] 

The result about the index and the nullity of $\Phi$ follows by a careful analysis of the sign of $\rho_\pm(k,p)$.

 In the case of the embedding $\Phi_0$, it is clear that $N$ and $iN$ project on $\mathbb{RP}^3$ and they give also a global orthonormal reference of the normal bundle of $\Phi_0$. Hence, the eigenvalues of its Jacobi operator are those of the Jacobi operator of $\Phi$ with $k$ even. From here the result for the index and nullity of $\Phi_0$ follows easily.
\end{proof}

\begin{proposition}\label{prop:index-nullity-totally-geodesic-spheres}
	Let $\Phi:\mathbb{S}^d\rightarrow \mathbb{S}^{2n+1}_\tau$ be the totally geodesic embedding given in \autoref{prop:classification-totally-geodesic}.(ii) . Then:
	\begin{enumerate}[(\scshape i)]
		\item $\Ind(\mathbb{S}^d) =  \begin{cases}
			2n+1+\frac{d(d-1)}{2} & \text{if } \tau^2 < 1, \\
			2n+1-d & \text{if } \tau^2 = 1.
		\end{cases}$
		
		\item $\Nul(\mathbb{S}^d) =  \begin{cases}
			(d+1)(2n+1-\frac{3d}{2}) & \text{if } \tau^2 < 1, \\
			(d+1)(2n+1-d) & \text{if } \tau^2 = 1.
		\end{cases}$
	\end{enumerate}
\end{proposition}

\begin{proof}
  The result for $\tau^2 = 1$ is well-known, see \cite{Simons68}, so we will assume that $\tau^2 < 1$.

  Let $\mathbb{S}^d$, $d \leq n$, be the totally geodesic sphere given in \autoref{prop:classification-totally-geodesic}.(ii). We recall that $\xi^\top = 0$ and $\mathbb{S}^d$ is totally real in $\mathbb{CP}^{n+1}(4(1-\tau^2))$. Then, its normal bundle can be orthogonally decomposed as 
  \[
    T^\perp \mathbb{S}^d = J(T\mathbb{S}^d) \oplus \langle \xi \rangle \oplus \mathcal{D}.
  \]
  Firstly, it is clear that $\{a_j\colon 1\leq j\leq 2(n-d)\}$ defined by
  \[
  a_j=(0,\dots,0,\stackrel{2(d+1)+j}{1},0,\dots,0)\in\mathbb{R}^{2n+2}
  \]
  is a global orthonormal reference of the bundle $\mathcal{D}$. Hence any section of the bundle $\mathcal{D}$ can be written as $\eta=\sum_{j=1}^{2(n-d)}f_j a_j$ with $f_j\in \mathcal{C}^{\infty}(\mathbb{S}^d)$.

  Now, since $\nablar_u a_j = 0$ and $a_j$  is orthogonal to $\xi$, we get from \eqref{eq:Levi-Civita-connection-Berger-round} that $\nabla^\perp_u a_j = 0$. As a consequence $\Delta^\perp a_j  = 0$ and so, using \eqref{eq:Jacobi-operator},
  \[
    \mathcal{L}\eta = \sum_{j = 1}^{2(n-d)}  (\Delta f_j + d f_j)a_j.
  \] 
  If $\Gamma(\mathcal{D})$ is the space of sections of the subbundle $\mathcal{D}$, then $\mathcal{L}(\Gamma(\mathcal{D}))\subset\Gamma(\mathcal{D})$.
  Moreover, since the first eigenvalues of $\Delta$ are $0$ and $d$ with multiplicities $1$ and $d+1$ (see   \autoref{lm:connection-Berger}.(4)) we get that $\Ind(\mathcal{L}\lvert_{\Gamma(\mathcal{D})}) = 2(n-d)$ and $\Nul(\mathcal{L}\lvert_{\Gamma(\mathcal{D})}) = 2(n-d)(d+1)$.
  
  Secondly, given $X \in \mathfrak{X}(\mathbb{S}^d)$ then $JX$ is normal and, if $\nablap$ is the Levi-Civita connection of $\mathbb{CP}^{n+1}(4(1-\tau^2))$, we have that $\nablap_uJX=J\nablap_uX$ for any tangent vector $u\in T\mathbb{S}^d$. Taking normal components in this equation, using Gauss and Weingarten formulae and~\eqref{eq:2ff-Berger-complex-space}, we get that
  \begin{equation}\label{eq:totally-geodesic-round-sphere-nabla-JX}
    \nabla^\perp_u JX = J(\nabla_u X) +J\hat{\sigma}(u,X)=J(\nabla_u X)-\tau \bprodesc{u}{X}\xi.
  \end{equation} 
 Now, thanks to \eqref{eq:totally-geodesic-round-sphere-nabla-JX} we easily get that
 \[
   \Delta^\perp JX = J (\Delta X - \tau^2 X) - 2\tau (\diver X) \xi.
 \] 
 Using this equation and  \eqref{eq:Jacobi-operator}, we get that, for any $f \in \mathcal{C}^{\infty}(\mathbb{S}^d)$,
 \begin{equation}\label{eq:Jacobi-operator-totallyreal}
   \mathcal{L}(JX + f\xi) = J[\Delta X + (d + 3 - 4\tau^2) X + 2\tau \nabla f] + [\Delta f - 2\tau \diver X]\xi,
 \end{equation}
 where we have used that $\xi$ is a Jacobi field, i.e., $\mathcal{L}\xi=0$.
 
 Now, we consider the identification $\Gamma(JT\mathbb{S}^d\oplus\langle\xi\rangle)\equiv \Omega^1(\mathbb{S}^d)\oplus \mathcal{C}^{\infty}(\mathbb{S}^d)$ given by
 \[
 JX+f\xi\equiv (\alpha,f),
 \]
 where $\alpha$ is the $1$-form on $\mathbb{S}^d$ given by $\alpha(Y)=\langle X,Y\rangle$, for any $Y\in\mathfrak{X}(\mathbb{S}^d)$ and $\Omega^p(\mathbb{S}^d)$ denotes the space of $p$-forms on $\mathbb{S}^d$. If $\hat{\Delta}$ represents the Hodge Laplacian acting on forms of $\mathbb{S}^d$, then the vector field $\Delta X-\Ric(X)=\Delta X-(d-1)X$ corresponds with the $1$-form $\hat{\Delta}\alpha$, and so the Jacobi operator $\mathcal{L}$ acting on the sections of the bundle $JT\mathbb{S}^d\oplus\langle\xi\rangle$ computed in \eqref{eq:Jacobi-operator-totallyreal} becomes in an operator $L:\Omega^1(\mathbb{S}^d)\oplus \mathcal{C}^{\infty}(\mathbb{S}^d)\rightarrow\Omega^1(\mathbb{S}^d)\oplus \mathcal{C}^{\infty}(\mathbb{S}^d)$ given by
 \[
 L(\alpha,f)=(\hat{\Delta}\alpha+2(d+1-2\tau^2)\alpha +2\tau \mathrm{d}f, \Delta f-2\tau\delta\alpha),
 \]
 where $\mathrm{d}$ is the differential and $\delta$ is the codifferential operator on $\mathbb{S}^d$ given by $\delta \alpha = \sum_{i = 1}^{d} (\nabla_{e_i} \alpha)(e_i)$ being $\{e_1,\ldots,e_d\}$ a local orthonormal reference on $\mathbb{S}^d$. It is clear that $\Ind (L)$ and $\Nul(L)$ are	 the index and the nullity of the Jacobi operator $\mathcal{L}$ acting on $\Gamma(JT\mathbb{S}^d\oplus\langle\xi\rangle)$.

Now, as the first Betti number of $\mathbb{S}^d$ is $0$, the Hodge decomposition theorem says that
\[
\Omega^1(\mathbb{S}^d)=\mathrm{d}\mathcal{C}^{\infty}(\mathbb{S}^d)\oplus\delta\Omega^2(\mathbb{S}^d),
\]
which allows to write in a unique way any $1$-form $\alpha$ as $\alpha=\mathrm{d}g+\delta\omega$, with $g\in \mathcal{C}^{\infty}(\mathbb{S}^d)$ and $\omega\in\Omega^2(\mathbb{S}^d)$. Now we can split the operator $L$ as $L=L_1\oplus L_2$ where
\begin{equation}\label{eq:index-totally-geodesic-sphere-L1}
\begin{gathered}
L_1: \mathrm{d} \mathcal{C}^{\infty}(\mathbb{S}^d)\oplus \mathcal{C}^{\infty}(\mathbb{S}^d)\rightarrow \mathrm{d}\mathcal{C}^{\infty}(\mathbb{S}^d)\oplus \mathcal{C}^{\infty}(\mathbb{S}^d)\\
L_1(\mathrm{d}g,f)=\bigl(\mathrm{d}(\Delta g+2(d+1-2\tau^2)g+2\tau f)), \Delta (f-2\tau g)\bigr),
\end{gathered}
\end{equation}
and 
\begin{equation}\label{eq:index-totally-geodesic-sphere-L2}
\begin{gathered}
L_2:\delta\Omega^2(\mathbb{S}^d)\rightarrow\delta\Omega^2(\mathbb{S}^d)\\
L_2(\delta\omega,0)=\bigl(\delta(\hat{\Delta}\omega+2(d+1-2\tau^2)\omega)),0\bigr).
\end{gathered}
\end{equation}
It is clear that $\Ind(L) = \Ind(L_1) + \Ind(L_2)$ and $\Nul(L) = \Nul(L_1) + \Nul(L_2)$. We now compute the index and nullity of both operators.

Let $\rho \leq 0$ be an eigenvalue of $L_1$ with associated eigenfunction $(\mathrm{d}g, f)$. Then, from the equality $L_1(\mathrm{d}g, f) + \rho(\mathrm{d}g, f) = 0$ and the expression of $L_1$ in \eqref{eq:index-totally-geodesic-sphere-L1} we deduce that
\begin{equation}\label{eq:index-totally-geodesic-sphere-L1-decomposition}
\begin{split}
  0 &= \mathrm{d}[\Delta g + \bigl(\rho + 2(d + 1 - 2\tau^2)\bigr)g + 2\tau f], \\
0 &= \Delta f - 2\tau \Delta g + \rho f.
\end{split}
\end{equation}
Notice that if $\mathrm{d}g = 0$ then from the first equation $f$ is constant and so, from the second one, $\rho = 0$. Hence, $(0, a)$, $a \in \mathbb{R}$, is in the nullity of $L_1$. 

If $\mathrm{d}g \neq 0$, we decompose $f = \sum_{k \geq 0} f_k$ and $g = \sum_{k \geq 0} g_k$ in eigenfunctions of the Laplacian $\Delta$ of $\mathbb{S}^d$, so $\Delta f_k + \lambda_k f_k = 0$ and $\Delta g_k + \lambda_k g_k = 0$ for all $k \geq 0$ where $\lambda_k = k(d+k-1)$ are the eigenvalues of the Laplacian on $\mathbb{S}^d$. Notice that $f_0$ and $g_0$ are constant functions. Hence, \eqref{eq:index-totally-geodesic-sphere-L1-decomposition} now reads
\begin{equation}\label{eq:index-totally-geodesic-sphere-L1-fk-gk}
\begin{split}
  0 & =[\rho + 2(d+1-2\tau^2) - \lambda_k] \mathrm{d}g_k + 2\tau \mathrm{d} f_k, \\ 
  0 &= 2\tau \lambda_k g_k + (\rho - \lambda_k)f_k,
\end{split}
\qquad k \geq 0.
\end{equation}
We now substitute $f_k$ from the second equation in the first one to obtain
\begin{equation}\label{eq:index-totally-geodesic-sphere-L1-dgk}
  \left[ \rho + 2(d + 1 - 2\tau^2)  - \lambda_k + \tfrac{4\tau^2 \lambda_k}{\lambda_k - \rho}\right] \mathrm{d}g_k = 0, \qquad k \geq 1.
\end{equation} 
And so, since we have assumed that $\mathrm{d}g \neq 0$, the eigenvalue $\rho \leq 0$ takes the form
\[
  \rho = \lambda_k - (d+1-2\tau^2) - \sqrt{(d+1-2\tau^2)^2 + 4\tau^2 \lambda_k}.
\] 
for some $k \geq 1$.  Moreover, its associated eigenfunction is $\bigl(\mathrm{d}g, \tfrac{2\tau \lambda_k}{\lambda_k - \rho_k} g\bigr)$, where $g$ is an eigenfunction of the Laplacian associated to $\lambda_k$. 

Now, $\rho < 0$ if and only if $\lambda_k < 2(d+1)$ which only occurs if $k= 1$. As a consequence $\Ind(L_1) = d+1$. If $\rho = 0$ then we get from~\eqref{eq:index-totally-geodesic-sphere-L1-fk-gk}
\[
  f_k = 2\tau g_k,\qquad [2(d+1) - \lambda_k] \mathrm{d}g_k = 0, \qquad k \geq 1.
\] 
so, since $\mathrm{d}g \neq 0$, we get that $g = g_2$ is an eigenfunction of the Laplacian associated to $\lambda_2 = 2(d+1)$ and $f = 2\tau g$. Hence, we obtain $\Nul(L_1) = \frac{1}{2}(d+2)(d+1)$ which is exactly the multiplicity of the eigenvalue $\lambda_2$ plus $1$ since we have previously analysed that $(0,a)$, $a \in \mathbb{R}$, are also in the nullity of $L_1$ and corresponds to the case $\mathrm{d}g = 0$.

Finally, we compute the index and nullity of $L_2$. From~\cite{IK1979}, it follows that only the eigenvalue $2(d-1)$ of $\hat{\Delta}$ acting on $\Omega^2(\mathbb{S}^d)$ provides a non-positive eigenvalue of $L_2$,  which is $-4(1-\tau^2)$. As its multiplicity is $\frac{1}{2}d(d+1)$, we obtain that $\Ind(L_2)=\frac{1}{2}d(d+1)$ and $\Nul(L_2)=0$.

From the previous analysis of the index and nullity of $\mathcal{L}\lvert_{\Gamma(\mathcal{D})}$, $L_1$ and $L_2$ we get the result.
\end{proof}

\begin{proposition}\label{prop:Clifford-hypersurfaces-index-nullity}
  Let $\Phi: \mathbb{S}^{2m_1+1}(r_1) \times \mathbb{S}^{2m_2+1}(r_2) \to \mathbb{S}^{2n+1}_\tau$, with $m_1+m_2+1 = n$, $r_1^2=\frac{2m_1+1}{2n}$, $r_2^2=\frac{2m_2+1}{2n}$, be a Clifford hypersurface described in \autoref{ex:Clifford-hypersurfaces}. Then 
 \begin{enumerate}[(i)]
   \item $\Ind(\Phi) = \begin{cases}
       1 & \text{if } \tau^2 \leq \frac{1}{2n+1},\\
        2n+3 & \text{if } \frac{1}{2n+1} < \tau^2 \leq 1.
   \end{cases}$

 \item $\Nul(\Phi) = \begin{cases}
     2(m_1+1)(m_2+1) & \text{if } \tau^2 \neq \frac{1}{2n+1}, \\
     2(m_1+1)(m_2+1) + 2(n+1) & \text{if } \tau^2 = \frac{1}{2n+1}.
 \end{cases}$
 \end{enumerate}
\end{proposition}

\begin{proof}
  Our first goal is to show a similar relation to \eqref{eq:laplacian} between the Laplacian $\Delta$ of the induced metric by $\Phi$ and the Laplacian $\prescript{g}{}{\Delta}$ of the standard product metric of $\mathbb{S}^{2m_1+1}(r_1) \times \mathbb{S}^{2m_2+1}(r_2)$ so we can write down an expression for the eigenvalues of $\Delta$ following a similar argument as in \autoref{lm:connection-Berger}.(4). We will denote by $\nabla$, $\prescript{g}{}{\nabla}$ the Levi-Civita connections of the Clifford hypersurfaces induced by the Berger $\bprodesc{\cdot}{\cdot}$ and the standard product metric $g$ respectively.

  Firstly, notice that, for any tangent vector field $X$ and any smooth function $f$, $\prodesc{\prescript{g}{}{\nabla} f}{X} = X(f) = \bprodesc{\nabla f}{X}$. Hence, using \eqref{eq:defini-metrica Berger}, we easily get 
  \[
    \nabla f = \prescript{g}{}{\nabla} f + \tfrac{1-\tau^2}{\tau^2} (L_\xi f)\xi.
  \] 

  Secondly, by \autoref{ex:Clifford-hypersurfaces}, we can take orthonormal references on the Clifford hypersurface $\{e_1,\dots,e_{n-1},\xi\}$ (respectively $\{e_1,\dots,e_{n-1},\tau\xi\}$) with respect to the metric $\bprodesc{\cdot}{\cdot}$ (respectively with respect to the metric $g$). Now, using  \eqref{eq:Levi-Civita-connection-Berger-round} and the above relation between the gradients we easily deduce 
  \begin{equation}\label{eq:Clifford-hypersurface-relation-Laplacian}
    \Delta f = \prescript{g}{}{\Delta} f + (1-\tau^2) (L_\xi)^2 f.
  \end{equation} 

  Now, it is well-known that each eigenvalue $\lambda$ of $\prescript{g}{}{\Delta}$ is of the form $\lambda = \lambda_{k_1} + \lambda_{k_2}$, where $\lambda_{k_j} = \frac{1}{r_j^2}k_j(2m_j+k_j)$ is the $k_j$-th eigenvalue of the Laplacian of $(\mathbb{S}^{2m_j+1}(r_j),g)$. The associated eigenspace $V(\lambda_{k_1} + \lambda_{k_2})$ to $\lambda_{k_1} + \lambda_{k_2}$ is generated by $f_1(p)\cdot f_2(q)$ for eigenfunctions $f_j$ associated to the eigenvalue $\lambda_{k_j}$. More precisely, $f_j$ is the restriction to $\mathbb{S}^{2m_j+1}(r_j)$ of a homogeneous harmonic polynomial of degree $k_j$. Hence, $f_1\cdot f_2$ is the restriction to $\mathbb{S}^{2m_1+1}(r_2) \times \mathbb{S}^{2m_2+1}(r_2)$ of a homogeneous harmonic polynomial of degree $k_1 + k_2$. As a consequence, so it is any eigenfunction $h \in V(\lambda_{k_1} + \lambda_{k_2})$. 

  Since $L_\xi$ and $\prescript{g}{}\Delta$ commutes and $L_\xi$ is skew-symmetric we get that $(L_\xi)^2$ is a self-adjoint linear transformation of $V(\lambda_{k_1}+\lambda_{k_2})$ with positive real eigenvalues. Then, we can decompose $h = \sum_{p} h_p$, with $h_p \in V(\lambda_{k_1}+\lambda_{k_2})$ eigenfunctions of $(L_\xi)^2$. But, thanks to the previous paragraph and following \cite[Lemma~3.1]{Tanno1979}, 
  \begin{equation}\label{eq:Clifford-hypersurface-L2xi}
    (L_\xi)^2h_p + \tfrac{1}{\tau^2}(k_1+k_2 - 2p)^2 h_p = 0, \quad \text{for certain } 0 \leq p \leq \lfloor \tfrac{1}{2}(k_1+k_2) \rfloor.
  \end{equation}
Therefore, using \eqref{eq:Clifford-hypersurface-relation-Laplacian}, each eigenvalue $\mu_{k_1,k_2,p}$ of $\Delta$ takes the form
\begin{equation}\label{eq:Clifford-hypersurface-eigenvalues-Laplacian}
  \mu_{k_1,k_2,p} = 2(m_1+m_2+1)\left( k_1 \tfrac{2m_1+k_1}{2m_1+1} + k_2 \tfrac{2m_2+k_2}{2m_2+1} \right) + \tfrac{1-\tau^2}{\tau^2}(k_1+k_2 - 2p)^2, 
\end{equation} 
for some integers $k_j \geq 0$ and $0 \leq p \leq \lfloor \tfrac{1}{2}(k_1 + k_2) \rfloor$, 
where we have taken into account the relation between the radii $r_1, r_2$ and $m_1, m_2$. 

Finally, from \eqref{eq:Jacobi-operator-hypersurfaces} and \autoref{ex:Clifford-hypersurfaces}, the Jacobi operator of the Clifford hypersurface is given by $\mathcal{L} = \Delta + 4(m_1+m_2+1)=\Delta+4n$. 
Then, thanks to \eqref{eq:Clifford-hypersurface-eigenvalues-Laplacian} the non-positive eigenvalues of $\mathcal{L}$ are:
\begin{enumerate}[(i)]
  \item $-4n <0$ with multiplicity $1$ (corresponds to $\mu_{0,0,0}$).
  \item $\frac{1}{\tau^2}-(2n + 1)<0$ with multiplicity $2(n+1)$ when $\tau^2 > \frac{1}{2n+1}$ (corresponds to $\mu_{1,0,0}$ and $\mu_{0,1,0}$).
  \item $0$ with multiplicity $2(n+1)$ when $\tau^2 = \frac{1}{2n+1}$ (corresponds to $\mu_{1,0,0}$ and $\mu_{0,1,0}$)
  \item $0$ with multiplicity $2(m_1+1)(m_2+1)$ (corresponds to $\mu_{1,1,1}$).
\end{enumerate}
where the multiplicities are obtained from the fact that the multiplicities of the first two eigenvalues of $(\mathbb{S}^{2m+1}(r),g)$ are given by $1$ and $2m+2$ respectively and a careful analysis of the condition~\eqref{eq:Clifford-hypersurface-L2xi} in the case of $\mu_{1,1,1}$. So the proof follows.
\end{proof}

\section{Stable compact minimal submanifolds of $\mathbb{S}^{2n+1}_{\tau}$}\label{sec:stable-compact-minimal-submanifolds-Berger-spheres}
Propositions~\ref{prop:index-nullity-totally-geodesic-Berger-spheres} and \ref{prop:index-nullity-Veronese} compute the index of the  $\mathbb{S}^1$-bundles compatible with the Hopf fibration  over the totally geodesic $\mathbb{CP}^m(4)\subset\mathbb{CP}^n(4)$ and the Veronese surface $\mathbb{CP}^1(2)\subset\mathbb{CP}^2(4)$, showing that they are stable when $\tau^2\leq \frac{1}{s(d+1)}$, where $d$ and $s$ are respectively  the dimension and the order of the submanifold. 

In the next result we generalize this property to any minimal submanifold of $\mathbb{S}^{2n+1}_{\tau}$ which is a  $\mathbb{S}^1$-bundle compatible with the Hopf fibration over a complex submanifold of $\mathbb{CP}^n(4)$.

\begin{theorem}\label{thm:S1-bundle-are-stable}
  Let $\Phi: M^{2m+1} \to \mathbb{S}^{2n+1}_\tau$ be a $\mathbb{S}^1$-bundle over a complex immersion $\Psi: N^{2m} \to \mathbb{CP}^n(4)$ compatible with the Hopf fibration and order $s$. If $0 <\tau^2\leq \frac{1}{s(2m+2)}$, then $\Phi$ is stable. 
\end{theorem}

\begin{remark}
In the case $\Phi$ is the induced $\mathbb{S}^1$-bundle then $s = 1$ and so it is stable if  $\tau^2 \leq \frac{1}{2m+2}$.
\end{remark}

\begin{proof}
Following the proof of \autoref{prop:induced-index}, $TM=\mathcal{D}\oplus\langle\xi\rangle$, $\mathcal{D}$ and $T^{\perp}M$ are invariant under the complex structure $J$ of $\mathbb{CP}^{n+1}(4(1-\tau^2))$ and the quadratic form associated to the Jacobi operator is 
  \begin{equation}\label{eq:quadratic-form-S1-bundle-over-complex}
    \mathcal{Q}(\eta) = \int_M \lvert\nabla^\perp \eta\rvert^2 - \bprodesc{\mathcal{A}\eta}{\eta} - (2m + \tau^2)\lvert\eta\rvert^2\, \mathrm{d}v.
  \end{equation}

We consider, for any $X \in \Gamma(\mathcal{D})$, the operator $D_X:\mathfrak{X}^\perp(M) \to \mathfrak{X}^\perp(M)$ given by
\[
D_X \eta = \nabla^\perp_{JX} \eta - J \nabla^\perp_X \eta.  
\]
Let $\{e_1,\ldots, e_m, Je_1,\ldots,Je_m, \xi\}$ be a local tangent orthonormal reference to $M$. Then
\[
\begin{split}
	\sum_{i = 1}^{m} \lvert D_{e_i}\eta\rvert^2 &= \sum_{i = 1}^{m}\{ \lvert\nabla^\perp_{Je_i}\eta\rvert^2 + \lvert J\nabla^\perp_{e_i}\eta\rvert^2 - 2 \bprodesc{\nabla^\perp_{Je_i}\eta}{J\nabla^\perp_{e_i}\eta}\} \\
	&= \lvert\nabla^\perp \eta\rvert^2 - \lvert\nabla^\perp_\xi \eta\rvert^2 - 2\sum_{i = 1}^{m} \bprodesc{\nabla^\perp_{Je_i}\eta}{J \nabla^\perp_{e_i}\eta}.
\end{split}
\] 
Now, let consider, for any $\eta \in \mathfrak{X}^\perp(M)$, the $1$-form on $M$ defined by $\alpha(u) =  \bprodesc{\nabla^\perp_{J(u - \bprodesc{u}{\xi}\xi)} \eta}{J\eta}$. Then, taking into account ~\eqref{eq:properties-covariant-derivative-2ff-immersion-xi-tangent-J-normal} and ~\eqref{eq:properties-immersion-xi-tangent-J-normal}, its codifferential is
\[
\begin{split}
	\delta \alpha &= \sum_{i = 1}^{m} \{ e_i(\alpha(e_i)) - \alpha(\nabla_{e_i}e_i) + (Je_i)(\alpha(Je_i)) - \alpha(\nabla_{Je_i}Je_i) + \xi(\alpha(\xi)) - \alpha(\nabla_\xi \xi) \} \\
	&= \sum_{i = 1}^{m} \{\bprodesc{\nabla^\perp_{e_i}\nabla^\perp_{Je_i} \eta}{J\eta} - \bprodesc{\nabla^\perp_{Je_i}\nabla^\perp_{e_i}\eta}{J\eta} + 2 \bprodesc{\nabla^\perp_{Je_i}\eta}{J\nabla^\perp_{e_i}\eta} - \alpha(\nabla_{e_i} e_i) - \alpha(\nabla_{Je_i}Je_i)\} \\
	&= \sum_{i = 1}^{m}\{ R^\perp(e_i, Je_i, \eta, J\eta) + \bprodesc{\nabla^\perp_{[e_i,Je_i]}\eta}{J\eta}+2 \bprodesc{\nabla^\perp_{Je_i}\eta}{J\nabla^\perp_{e_i}\eta} 
	- \alpha(\nabla_{e_i} e_i) - \alpha(\nabla_{Je_i}Je_i)\}, \\
\end{split}
\] 
where $R^{\perp}$ is the normal curvature of the immersion $\Phi$.
 Now, again from \eqref{eq:properties-immersion-xi-tangent-J-normal} we easily get that
\[
\sum_{i = 1}^{m} \bprodesc{\nabla^\perp_{[e_i,Je_i]}\eta}{J\eta}- \alpha(\nabla_{e_i} e_i) - \alpha(\nabla_{Je_i}Je_i) = -2\tau m \bprodesc{\nabla^\perp_\xi \eta}{J\eta}.
\] 
Moreover, from Ricci equation and using~\eqref{eq:Riemann-tensor-Berger} and~\eqref{eq:properties-immersion-xi-tangent-J-normal} 
\[
\begin{split}
	\sum_{i = 1}^{m} R^\perp(e_i, Je_i, \eta, J\eta) &= \sum_{i = 1}^{m} \bR(e_i, Je_i, \eta, J\eta) + \bprodesc{[A_\eta, A_{J\eta}]e_i}{Je_i} \\
	&= -2m(1-\tau^2) \lvert\eta\rvert^2 - \bprodesc{\mathcal{A}\eta}{\eta}.
\end{split}
\]
As a consequence of all the above computations we obtain that
\[
\sum_{i = 1}^{m} \lvert D_{e_i}\eta\rvert^2 = \lvert\nabla^\perp\eta\rvert^2 - \bprodesc{\mathcal{A}\eta}{\eta} - \lvert\nabla^\perp_\xi \eta\rvert^2 - 2m(1-\tau^2)\lvert\eta\rvert^2 - 2\tau m\bprodesc{\nabla^\perp_\xi \eta}{J\eta} - \delta\alpha.
\] 
Using the last expression in \eqref{eq:quadratic-form-S1-bundle-over-complex},  the quadratic form $\mathcal{Q}$ over any normal vector field $\eta$  can be written as
\begin{equation}\label{eq:inequality-Q-for-S1-bundle}
  \begin{split}
    \mathcal{Q}(\eta) &= \int_M \left(\sum_{i = 1}^{m} \lvert D_{e_i}\eta\rvert^2 + \lvert\nabla^\perp_\xi \eta\rvert^2 - \tau^2(2m+1) \lvert\eta\rvert^2 + 2\tau m \bprodesc{\nabla^\perp_\xi \eta}{J\eta}\right) \mathrm{d}v \\
            &\geq \int_M \left(\lvert \nabla^\perp_\xi \eta\rvert^2 - \tau^2(2m+1) \lvert\eta\rvert^2 + 2\tau m \bprodesc{\nabla^\perp_\xi \eta}{J\eta}\right) \mathrm{d}v, \\
  \end{split}
\end{equation}
where we have used that $\lvert D_{e_i}\eta\rvert^2 \geq 0$.

Now, to get a lower bound of~\eqref{eq:inequality-Q-for-S1-bundle} we change the Berger metric $\langle,\rangle$ in $\mathbb{S}^{2n+1}_{\tau}$ by the standard one $g$. It is clear that the volumen forms $\mathrm{d}v$ of $(M, \Phi^* \bprodesc{\cdot}{\cdot})$  and $\mathrm{d}v_g$ of $(M, \Phi^*g)$ are related by $\mathrm{d}v = \tau \mathrm{d}v_g$. Also, thanks to~\eqref{eq:Levi-Civita-connection-Berger-round} and~\eqref{eq:properties-covariant-derivative-2ff-immersion-xi-tangent-J-normal}, we get
\begin{equation}\label{eq:S1-bundle-stable-nabla-normal-section}
\nabla^\perp_\xi \eta = \tfrac{1}{\tau} \left( \nablag^{\perp}_V \eta - (1-\tau^2) J\eta\right),\qquad A^g_\eta V = 0,
\end{equation}
where $A^g$ is the shape operator of $\Phi: M \to \mathbb{S}^{2n+1}$, and $V$ is the tangent vector field on $M$ defined by $V_p = ip,\,\forall p\in M$. 
 Hence, by the previous formula and \eqref{eq:killing} the inequality in \eqref{eq:inequality-Q-for-S1-bundle} becomes
\[
	\mathcal{Q}(\eta) \geq \tfrac{1}{\tau} \int_M \left[  \lvert\nablag^{\perp}_{V} \eta\rvert^2_g + (\tau^2(2m+2) - 2) \prodesc{\nablag^{\perp}_{V} \eta}{J\eta} + ( 1 - \tau^2 (2m+2)) \lvert\eta\rvert^2_g  \right] \mathrm{d}v_g.
\] 

\noindent\textbf{Claim}: \emph{The operator $G: \mathfrak{X}^\perp(M) \to \mathfrak{X}^\perp(M)$ given by 
  \[
  G \eta=\nablag_V^\perp \nablag_V^\perp \eta-\nablag_{\nablag_VV}^{\perp}\eta=\nablag_V^\perp \nablag_V^\perp \eta,
  \] 
   is a  self-adjoint operator with spectrum $\{\frac{k^2}{s^2}\colon k\in\mathbb{Z},\,k\geq 0\}$.}

Now, we can write $\eta = \sum_{k \geq 0} \eta_k$, with $\eta_k$ an eigensection of $G$ associated to the eigenvalue $\frac{k^2}{s^2}$, i.e., $G\eta_k+\frac{k^2}{s^2}\eta_k=0$. In particular $\eta_0$ satisfies $\nablag^\perp_V \eta_0 = 0$. Hence from the last expression for the quadratic form $Q$ and as $G(\nablag^\perp_V \eta_k)+\frac{k^2}{s^2}\nablag^\perp_V \eta_k=0$ we get
\begin{multline}\label{eq:lower-bound-quadratic-form-xi-tangent}
\mathcal{Q}(\eta) \geq \tfrac{1}{\tau} \int_M (1 - \tau^2(2m+2)) \lvert\eta_0\rvert^2_g \,\mathrm{d}v_g\\ 
+ \tfrac{1}{\tau} \sum_{k \geq 1} \int_M \left[  \lvert\nablag^{\perp}_{V} \eta_k\rvert^2_g + \bigl(\tau^2(2m+2) - 2\bigr) \prodesc{\nablag^{\perp}_{V} \eta_k}{J\eta_k} + \bigl(1 - \tau^2 (2m+2)\bigr) \lvert\eta_k\rvert^2_g  \right] \mathrm{d}v_g\\
\geq \tfrac{1}{\tau} \sum_{k \geq 1} \int_M \left[\left(\frac{k^2}{s^2}+1 - \tau^2 (2m+2)\right) \lvert\eta_k\rvert^2_g +   \bigl(\tau^2(2m+2) - 2\bigr) \prodesc{\nablag^{\perp}_{V} \eta_k}{J\eta_k}   \right] \mathrm{d}v_g,
\end{multline}
where we have used that $1-\tau^2(2m+2)\geq 0$. 

 Now, for any $k\geq 1$ we have that 
\[
0 \leq \frac{1}{2}\int_M \lvert\nablag^\perp_V \eta_k -\tfrac{k}{s} J\eta_k\rvert^2_g \, \mathrm{d}v_g = \frac{k^2}{s^2} \int_M  \lvert\eta_k\rvert^2_g \,\mathrm{d}v_g - \frac{k}{s} \int_M \prodesc{\nablag^\perp_V \eta_k}{J\eta_k} \,\mathrm{d}v_g.
\] 
As $\tau^2(2m+2)-2<0$, using the above inequality in~\eqref{eq:lower-bound-quadratic-form-xi-tangent} we obtain
\[
  \begin{split}
    Q(\eta) &\geq   
            \tfrac{1}{\tau} \sum_{k \geq 1} \left( \frac{k}{s}-1\right)\left[ \left( \frac{k}{s} -1 \right) +2\tau^2(m+1) \right] \int_M \lvert\eta_k\rvert^2_g \\
          &\geq \tfrac{1}{\tau} \sum_{k= 1}^{s-1} \left( \frac{k}{s}-1\right)\left[ \left( \frac{k}{s} -1 \right) +2\tau^2(m+1) \right] \int_M \lvert\eta_k\rvert^2_g \geq 0,       
  \end{split}
\] 
because for $1 \leq k \leq s-1$ we have that $(\frac{k}{s}-1)2\tau^2(m+1)\geq (\frac{k}{s}-1)\frac{1}{s}$.  Hence $\Phi$ is stable.

 \noindent\textbf{Proof of the claim}: 
  Firstly, since $\diver_g V = 0$,
  \[
\begin{split}
  0 &= \int_M \diver_g (\prodesc{\nablag^\perp_V \eta}{\zeta} V) \mathrm{d}v_g = \int_M \left[\prodesc{\nablag^\perp_V \nablag^\perp_V \eta}{\zeta} + \prodesc{\nablag^\perp_V \eta}{\nablag^\perp_V \zeta}\right] \mathrm{d}v_g,\\
    &= \int_M g(G\eta, \zeta) + \prodesc{\nablag^\perp_V \eta}{\nablag^\perp_V \zeta}\, \mathrm{d}v_g,
\end{split} 
 \] 
and  so $\int_Mg(G\eta,\zeta)\,\mathrm{d}v_g=\int_Mg(\eta,G\zeta)\,\mathrm{d}v_g$, which proves that $G$ is self-adjoint.

In fact, $G$ is the vertical normal Laplacian with respect to the fibration $\hat{\pi}:M^{2m+1}\rightarrow N^{2m}$, because if $p\in M$ and $F_p$ is the fiber through $p$, then
\[
(G\eta)_{\lvert F_p}=\Delta_p^{\perp}(\eta_x{\lvert F_p}),
\] 
where $\Delta_p^{\perp}$ is the normal Laplacian of the totally geodesic immersion $F_p\subset M^{2m+1}\rightarrow \mathbb{S}^{2n+1}$. As the order of $M$ is $s$,  all these immersions are congruent to the immersion $:\nu:\mathbb{S}^1\rightarrow\mathbb{S}^{2n+1}$ given by $z\mapsto (z^s,0\dots,0)$, and so  the eigenvalues of the normal Laplacian $\Delta_p^{\perp}$ will be those of the normal Laplacian of $\nu$. If $\Delta_{\nu}^{\perp}$ is the normal Laplacian of the immersion $\nu$, and $\{e_1,\dots,e_n,Je_1,\dots,Je_n\}$ is a global orthonormal reference of the normal bundle, then if $\nu=\sum_{i=1}^n\{f_ie_i+g_iJe_i\}$ is a normal section, then
\[
\Delta_{\nu}^{\perp}\nu=\sum_{i=1}^n\{(\Delta f_i)e_i+(\Delta g_i)Je_i\}.
\]
Hence the eigenvalues of $\Delta_{\nu}^{\perp}$ will be those of the Laplacian of $\mathbb{S}^1$ endowed with the induced metric by $\nu$, i.e., the set $\{\frac{k^2}{s^2}:k\in\mathbb{Z},k\geq 0\}$.
\end{proof}

Since any compact minimal submanifold of the sphere $(\mathbb{S}^{2n+1},g)$ is unstable, any compact minimal submanifold of $\mathbb{S}^{2n+1}_{\tau}$ is also unstable for $\tau$ next to $1$. In the following result we obtain the first value of $\tau$ for which the instability disappears. This value is $\tau^2=\frac{1}{d+1}$, where $d$ is the dimension of the submanifold and for this $\tau$ we classify the stable compact minimal embedded submanifolds.  This value of $\tau$ can be also interpreted as the first value of $\tau$ for which the minimal submanifold $M^d$ of $ \mathbb{S}^{2n+1}_{\tau}$ is also minimal in $\mathbb{CP}^{n+1}(4(1-\tau^2))$. In fact,  from ~\eqref{eq:2ff-Berger-complex-space}, the mean curvature $H$ of $M^d$ in $\mathbb{CP}^{n+1}(4(1-\tau^2))$ is given by
\[
  H = \tfrac{1}{d}\left(\tau d-\tfrac{1-\tau^2}{\tau}\lvert\xi^{\top}\rvert^2\right)J\xi.
\]
Hence $H=0$ if and only if $\tau^2 d =(1-\tau^2)\lvert \xi^{\top}\rvert^2$. As $\lvert\xi^{\top}\rvert^2\leq 1$, if $H=0$ then $\tau^2\leq\frac{1}{d+1}$. It is clear that if $\tau^2=\frac{1}{d+1}$ and $\xi^{\perp}=0$ then $H=0$.

\begin{theorem}\label{thm:classification-stable}
  Let $\Phi: M^d \to \mathbb{S}^{2n+1}_\tau$ be a minimal immersion of a compact $d$-manifold $M$ in the Berger sphere $\mathbb{S}^{2n+1}_\tau$. If $\frac{1}{d+1} \leq \tau^2 \leq 1$ and $\Phi$ is stable then $\tau^2 = \frac{1}{d+1}$, $d = 2m+1$ and $M$ is a $\mathbb{S}^1$-bundle $\hat{\pi}:M^{2m+1}\rightarrow N^{2m}$ over a complex submanifold  $\Psi: N^{2m} \to \mathbb{CP}^n(4)$ compatible with the Hopf fibration.
\end{theorem}

As a consequence of Theorems~\ref{thm:S1-bundle-are-stable}, \ref{thm:classification-stable},  and \autoref{prop:killing-tangente-normal}.(i) we get the following result:
\begin{corollary}
Let $\Phi: M^d \to \mathbb{S}^{2n+1}_\tau$ be a minimal embedding of a compact $d$-manifold $M$ in the Berger sphere $\mathbb{S}^{2n+1}_\tau$. If $\frac{1}{d+1} \leq \tau^2 \leq 1$, then  $\Phi$ is stable if and only if $\tau^2 = \frac{1}{d+1}$, $d = 2m+1$ and $M$ is the induced $\mathbb{S}^1$-bundle by the Hopf fibration over a complex embedded submanifold  $\Psi: N^{2m} \to \mathbb{CP}^n(4)$. 
\end{corollary}

\begin{remark}
  The authors believe that the minimal examples described in \autoref{prop:killing-tangente-normal}.(i) are unstable when $\tau^2=\frac{1}{2m+2}$ and the order $s \geq 2$, as \autoref{ex:examples-induced-bundle-by-Hopf-fibration}.(\ref{ex:sphere}) for $m = 0$ and \autoref{ex:examples-induced-bundle-by-Hopf-fibration}.(\ref{ex:Veronese}) corroborate (see Propositions~\ref{prop:index-nullity-totally-geodesic-Berger-spheres} and~\ref{prop:index-nullity-Veronese}).
\end{remark}

\begin{proof}
  \textbf{Claim 1}: \emph{If $\frac{1}{d+1}\leq \tau^2 \leq 1$ and $\Phi$ is stable then $\tau^2 = \frac{1}{d+1}$ and $\xi^\perp = 0$.}

Firstly, we are going to define certain vector fields on $\mathbb{S}^{2n+1}_{\tau}$, whose normal components will be test sections for the quadratic form $\mathcal{Q}$. To do that, for each $a \in \mathbb{C}^{n+1}$ let $X_a \in \mathfrak{X}(\mathbb{S}^{2n+1}) = \mathfrak{X}(\mathbb{S}^{2n+1}_\tau)$ be the vector field given by 
\[
  (X_a)_p = a - g(a,p)p,\quad \text{for all } p \in \mathbb{S}^{2n+1}.
\] 

If $\nablae$ is the Levi-Civita connection of the Euclidean metric $g$ in $\mathbb{C}^{n+1}$, then we get for any $u \in T_p \mathbb{S}^{2n+1}$,
	\[
	0 = \nablae_u a = \nablae_u(X_a + \prodesc{a}{p}p) = \nablar_u X_a + \prodesc{a}{p}u,
	\] 
  where we recall that $\nablar$ is the Levi-Civita connection of $(\mathbb{S}^{2n+1}, g)$ (see Lemma~\ref{lm:connection-Berger}).
	So $\nablar_u X_a = -f_au$, where $f_a: \mathbb{S}^{2n+1} \to \mathbb{R}$ is given by $f_a(p) = \prodesc{a}{p}$. 

Now, using the relation between the Levi-Civita connections $\nablab$ of $\mathbb{S}^{2n+1}_\tau$ and $\nablar$ given in~\eqref{eq:Levi-Civita-connection-Berger-round}, we get the following behaviour of the vector field $X_a$ with respect to $\nablab$
	\begin{equation}\label{eq:derivative-Xa}
		\bprodesc{\nablab_u X_a}{v} = -f_a \bprodesc{u}{v} - \tfrac{1 - \tau^2}{\tau}\left( \bprodesc{X_a}{\xi} \bprodesc{Ju}{v} + \bprodesc{u}{\xi} \bprodesc{JX_a}{v}\right).
	\end{equation} 
  for any $u, v \in T_p \mathbb{S}^{2n+1}_\tau$. 

  In this situation, given the immersion $\Phi: M^d \to \mathbb{S}^{2n+1}_\tau$, we decompose $X_a \in \mathfrak{X}(\mathbb{S}^{2n+1}_\tau)$ in its tangent a normal component to $\Phi$, i.e.\ $X_a= X_a^\top + X_a^\perp$. Our goal is to compute the Jacobi operator~\eqref{eq:Jacobi-operator} acting on $X_a^\perp$.

  From \eqref{eq:derivative-Xa} we deduce that, for any $u \in T_pM$, 
	\begin{align}
		\nabla_{u} X_a^\top &= -f_a u + A_{X_a^\perp} u - \tfrac{1-\tau^2}{\tau} \bigl[ \bprodesc{X_a}{\xi} (Ju)^\top + \bprodesc{u}{\xi} (JX_a)^\top \bigr] \label{eq:derivative-Xa-tangent}, \\
		\nabla^\perp_{u} X_a^\perp &= -\sigma(u, X_a^\top) - \tfrac{1-\tau^2}{\tau} \bigl[ \bprodesc{X_a}{\xi}(Ju)^\perp + \bprodesc{u}{\xi}(JX_a)^\perp \bigr] \label{eq:derivative-Xa-normal}.
	\end{align} 
	
	We will make the computation of $\mathcal{L}X_a^\perp$ at a point $p \in M$ and we will use a orthonormal tangent reference $\{e_1,\dots,e_d\}$ to $M$ satisfying $(\nabla_{e_k} e_j)_p = 0$. Hence, taking normal derivatives with respect to $e_i$ in \eqref{eq:derivative-Xa-normal} and using Codazzi equation we deduce
	\[
	\begin{split}
		\nabla_{e_i}^\perp \nabla_{e_i}^\perp X_a^\perp &= (\bR(X_a^\top, e_i)e_i)^\perp - \sigma(e_i, \nabla_{e_i}X_a^\top) - \tfrac{1-\tau^2}{\tau} \nabla^\perp_{e_i} \left( \bprodesc{X_a}{\xi}(Je_i)^\perp + \bprodesc{e_i}{\xi} (JX_a)^\perp \right).
	\end{split}
	\] 
	But, using \eqref{eq:derivative-Xa-tangent}, the minimality assumption and that $\sum_{i} \sigma(e_i, (Je_i)^\top) = 0$ because the skew-symmetry of $J$, we get
	\[
	\sum_{i = 1}^{d} \sigma(e_i, \nabla_{e_i}X_a^\top) = \mathcal{A}X_a^\perp - \tfrac{1-\tau^2}{\tau} \sigma(\xi^\top, (JX_a)^\top).
	\] 
	As a consequence
	\begin{equation}\label{eq:normal-Laplacian-Xa}
		\begin{split}
			\Delta^\perp X_a^\perp = &\sum_{i = 1}^{d} (\bR(X_a^\top,e_i)e_i)^\perp - \mathcal{A} X_a^\perp + \tfrac{1-\tau^2}{\tau} \sigma(\xi^\top, (JX_a)^\top) \\
			&- \tfrac{1-\tau^2}{\tau} \sum_{i = 1}^{d} \nabla^\perp_{e_i} \left( \bprodesc{X_a}{\xi}(Je_i)^\perp + \bprodesc{e_i}{\xi}(JX_a)^\perp \right).
		\end{split}
	\end{equation} 
	We now compute the last sum. Firstly, using the minimality assumption, \eqref{eq:covariant-derivative-xi} and~\eqref{eq:derivative-Xa}, we get 
	\begin{equation}\label{eq:auxiliar-derivative-Xa-xi}
		e_i(\bprodesc{X_a}{\xi}) = -f_a \bprodesc{e_i}{\xi} - \tau \bprodesc{JX_a}{e_i}\quad \text{and} \quad \sum_{i = 1}^{d} e_i(\bprodesc{e_i}{\xi}) = 0.
	\end{equation} 
	As a direct consequence of the previous equation
	\[
	\begin{split}
		\sum_{i = 1}^{d} \nabla^\perp_{e_i} \left( \bprodesc{X_a}{\xi}(Je_i)^\perp + \bprodesc{e_i}{\xi}(JX_a)^\perp \right) &= -f_a(J\xi^\top)^\perp - \tau[J(JX_a)^\top]^\perp \\
		&+ \sum_{i = 1}^{d} \left(\bprodesc{X_a}{\xi} \nabla^\perp_{e_i}(Je_i)^\perp + \bprodesc{e_i}{\xi} \nabla^\perp_{e_i} (JX_a)^\perp\right).
	\end{split}
	\] 
  Now, in order to compute the last term of the above formula we are going to obtain a general expression for $\nabla_u^\perp (JY)^\perp$ where $Y$ is any vector field in $\mathbb{S}^{2n+1}_\tau$ and $u \in T_pM$. To do so, we are going to consider $\mathbb{S}^{2n+1}_\tau$ isometrically embedded in $\mathbb{CP}^{n+1}(4(1-\tau^2))$ (see \autoref{prop:geodesic-spheres-complex-spaces}). On the one hand, decomposing $JY \in \mathfrak{X}(\mathbb{CP}^{n+1}(4(1-\tau^2)))$ as $JY = (JY)^\top + (JY)^\perp + \bprodesc{Y}{\xi}J\xi$ and using~\eqref{eq:2ff-Berger-complex-space} and~\eqref{eq:covariant-derivative-xi} we deduce
	\[
    \begin{split}
      (\nablap_u JY)^\perp &= (\nablap_u (JY)^\top)^\perp + (\nablap_u (JY)^\perp)^\perp + \bprodesc{Y}{\xi} (\nablap_u J\xi)^\perp \\
                           &= \sigma(u, (JY)^\top) + \nabla^\perp_u (JY)^\perp + \tfrac{1-\tau^2}{\tau} \bprodesc{Y}{\xi} \bprodesc{u}{\xi}\xi^\perp,
    \end{split}
	\] 
where $\nablap$ is the Levi-Civita connection of $\mathbb{CP}^{n+1}(4(1-\tau^2))$.

	On the other hand, using \eqref{eq:2ff-Berger-complex-space}
	\[
	(\nablap_u JY)^\perp = (J \nablap_u Y)^\perp = (J \nablab_u Y)^\perp - \left( \tau \bprodesc{u}{Y} - \tfrac{1-\tau^2}{\tau} \bprodesc{u}{\xi} \bprodesc{Y}{\xi} \right)\xi^\perp.
	\] 
	Therefore, from both expressions of $(\nablap_u JY)^\perp$, we get
	\begin{equation}\label{eq:derivative-JY-normal}
		\nabla^\perp_u (JY)^\perp = (J \nablab_u Y)^\perp -\sigma(u, (JY)^\top)  - \tau \bprodesc{u}{Y}\xi^\perp.
	\end{equation}
	As a consequence, taking into account the minimality of $\Phi$ and~\eqref{eq:derivative-JY-normal}, we get 
	\begin{equation}\label{eq:covariant-derivativa-Je-normal}
		\sum_{i = 1}^{d} \nabla^\perp_{e_i} (Je_i)^\perp = -\tau d \xi^\perp,
	\end{equation} 
	and also, by \eqref{eq:derivative-Xa} and \eqref{eq:derivative-JY-normal},
	\[
	\begin{split}
		\sum_{i = 1}^{d} \bprodesc{e_i}{\xi} \nabla^\perp_{e_i} (JX_a)^\perp = &-f_a (J\xi^\top)^\perp - 2 \tfrac{1-\tau^2}{\tau} \bprodesc{X_a}{\xi}\lvert\xi^\top\rvert^2 \xi^\perp + \tfrac{1-\tau^2}{\tau} \lvert\xi^\top\rvert^2 X_a^\perp \\
		&- \tau \bprodesc{X_a}{\xi^\top} \xi^\perp - \sigma(\xi^\top, (JX_a)^\top).
	\end{split}
	\] 
	Lastly, using \eqref{eq:Riemann-tensor-Berger} we get
	\begin{equation}\label{eq:Riemann-tensor-Xa-tangent}
		\sum_{i = 1}^{d} (\bR(X_a^\top,e_i)e_i)^\perp = - 3(1-\tau^2) [J(JX_a^\top)^\top]^\perp + (1-\tau^2)(1-d)\bprodesc{X_a^\top}{\xi}\xi^\perp.
	\end{equation} 
	Using the above computations in \eqref{eq:normal-Laplacian-Xa} and taking into account the expression of the Jacobi operator given in \eqref{eq:Jacobi-operator}, we finally obtain 
	\begin{equation}\label{eq:Jacobi-Xa-normal}
		\begin{split}
			\mathcal{L} X_a^\perp = & \left( d - \tfrac{1-\tau^2}{\tau^2} \lvert\xi^\top\rvert^2 \right) X_a^\perp +2 \tfrac{1-\tau^2}{\tau} \sigma(\xi^\top, (JX_a)^\top) + 2 \tfrac{1-\tau^2}{\tau} f_a (J\xi^\top)^\perp \\ 
			&- 2(1-\tau^2)[J(JX_a)^\top]^\perp +2(1-\tau^2) \left( \bprodesc{X_a}{\xi^\top} + \tfrac{1-\tau^2}{\tau^2} \lvert\xi^\top\rvert^2 \bprodesc{X_a}{\xi}\right) \xi^\perp. \\
		\end{split}
	\end{equation} 

  We will now compute $\sum_{j} \bprodesc{\mathcal{L} X_{a_j}^\perp}{X_{a_j}^\perp}$ for a $g$-orthonormal basis of $\mathbb{C}^{n+1}$ $\{a_j\colon j = 1,\ldots, 2n+2\}$. Taking into account the definition of the Berger metric \eqref{eq:defini-metrica Berger} we get
  \begin{equation}\label{eq:sum-product-Xa}
    \sum_{j = 1}^{2n+2} \bprodesc{X_{a_j}}{u} \bprodesc{X_{a_j}}{v} = \bprodesc{u}{v} - (1-\tau^2) \bprodesc{u}{\xi} \bprodesc{v}{\xi}. 
  \end{equation} 
Then, as a direct consequence of the previous formula,
\begin{align*}
  \sum_{j = 1}^{2n+2} \bprodesc{X_{a_j}^\perp}{X_{a_j}^\perp} &= \sum_{\alpha = 1}^q \sum_{j =1}^{2n+2} \bprodesc{X_{a_j}}{\zeta_\alpha}^2 =  q - (1-\tau^2) \lvert\xi^\perp\rvert^2, \\
  \sum_{j = 1}^{2n+2} f_{a_j} \bprodesc{(J\xi^\top)^\perp}{X_{a_j}} &= \sum_{j = 1}^{2n+2} \prodesc{a_j}{p}\left( \prodesc{(J\xi^\top)^\perp}{a_j} - (1-\tau^2) \prodesc{(J\xi^\top)^\perp}{ip} \prodesc{a_j}{ip}  \right) \\
  &= \prodesc{(J\xi^\top)^\perp}{p} - (1-\tau^2) \prodesc{(J\xi^\top)^\perp}{ip} \prodesc{p}{ip} = 0, \\
  \sum_{j = 1}^{2n+2} \bprodesc{J(JX_{a_j})^\top}{X_{a_j}^\perp} &= -\sum_{i = 1}^{d} \sum_{j = 1}^{2n+2} \bprodesc{X_{a_j}}{Je_i} \bprodesc{X_{a_j}}{(Je_i)^\perp} = -\sum_{i = 1}^{d} \lvert(Je_i)^\perp\rvert^2, \\
  \sum_{j = 1}^{2n+2} \bprodesc{X_{a_j}}{\xi^\top} \bprodesc{X_{a_j}}{\xi^\perp} &= -(1-\tau^2) \lvert\xi^\perp\rvert^2 \lvert\xi^\top\rvert^2,\\
  \sum_{j = 1}^{2n+2} \bprodesc{X_{a_j}}{\xi} \bprodesc{X_{a_j}}{\xi^\perp} &= \tau^2 \lvert\xi^\perp\rvert^2,
\end{align*}
where $\{\zeta_\alpha\colon \alpha = 1,\ldots, q\}$ is an orthonormal reference of the normal bundle of $\Phi$. Moreover, using again \eqref{eq:sum-product-Xa} and taking into account that $\sigma(\xi^\top, u) = \tau(Ju)^\perp - \nabla^\perp_u \xi^\perp$ from \eqref{eq:covariant-derivative-xi}, 
\[
\begin{split}
  \sum_{j = 1}^{2n+2} \bprodesc{\sigma(\xi^\top, (JX_{a_j})^\top)}{X_{a_j}^\perp} &= - \sum_{i,\alpha} \bprodesc{\sigma(\xi^\top, e_i)}{\zeta_\alpha}\sum_{j = 1}^{2n+2} \bprodesc{X_{a_j}}{Je_i} \bprodesc{X_{a_j}}{\zeta_\alpha} \\
  &= -\sum_{i = 1}^{d} \bprodesc{\sigma(\xi^\top, e_i)}{Je_i} = -\tau \sum_{i = 1}^{d} \lvert(Je_i)^\perp\rvert^2 + \sum_{i = 1}^{d} \bprodesc{\nabla^\perp_{e_i} \xi^\perp}{Je_i} = \\
  &= -\tau \sum_{i = 1}^{d} \lvert(Je_i)^\perp\rvert^2 + \tau d \lvert\xi^\perp\rvert^2 + \delta \beta,
\end{split}
\] 
where $\beta$ is the $1$-form $\beta(X) = \bprodesc{\xi^\perp}{JX}$ whose codifferential is given by 
\[
  \delta \beta = \sum_{i = 1}^{d} \left( \bprodesc{\nabla^\perp_{e_i}\xi^\perp}{Je_i} + \bprodesc{\xi^\perp}{\nabla^\perp_{e_i} (Je_i)^\perp} \right)  = \sum_{i = 1}^{d}  \bprodesc{\nabla^\perp_{e_i}\xi^\perp}{Je_i} -\tau d \lvert\xi^\perp\rvert^2,
\] 
thanks to \eqref{eq:covariant-derivativa-Je-normal}.

Finally, after a long but straightforward computation writing $\lvert\xi^\top\rvert^2 = 1 - \lvert\xi^\perp\rvert^2$,
  \[
 \  \sum_{j = 1}^{2n+2} \bprodesc{\mathcal{L} X_{a_j}^\perp}{X_{a_j}^\perp}= -P(\lvert\xi^\perp\rvert^2) + 2\tfrac{1-\tau^2}{\tau} \delta \beta,
  \] 
  where $P$ is the polynomial 
  \begin{equation}\label{eq:polynomial-Jacobi-operator}
    P(x) = \tfrac{(1-\tau^2)^2}{\tau^2} x^2 - \tfrac{1-\tau^2}{\tau^2}\bigl(1+q+\tau^2(d-1)\bigr) x -q\bigl(d + 1 - \tfrac{1}{\tau^2}\bigr).
  \end{equation} 

  Now, if $M$ is stable and $\frac{1}{d+1} \leq \tau^2 \leq 1$ then, using the divergence theorem, we have that 
  \[
    \begin{split}
      0 \leq \sum_{j = 1}^{2n+2} \mathcal{Q}(X_{a_j}^\perp) = \int_M P(\lvert\xi^\perp\rvert^2) &\leq \int_M -\tfrac{(1-\tau^2)}{\tau^2} (q + \tau^2 d) \lvert\xi^\perp\rvert^2 - q(d+1-\tfrac{1}{\tau^2}) \\
    &\leq -q(d+1-\tfrac{1}{\tau^2}) \mathrm{vol}(M) \leq 0,
    \end{split}
  \] 
  where we have used that $\lvert\xi^\perp\rvert^4 \leq \lvert\xi^\perp\rvert^2$ in the second inequality and that $\lvert\xi^\perp\rvert^2 \geq 0$ in the third one. Therefore, $\tau^2 = \frac{1}{d+1}$ and $\xi^{\perp}=0$.

  \textbf{Claim 2}: \emph{Let $\Phi: M^d \to \mathbb{S}^{2n+1}_\tau$ be a stable minimal immersion of a compact $d$-dimensional manifold. If $\tau^2 = \frac{1}{d+1}$ and $\xi^\perp = 0$ then  the normal bundle $T^{\perp}M$ and the subbundle $\mathcal{D}$ defined by  $TM=\mathcal{D}\oplus\langle\xi\rangle$ are invariants under $J$. } 

  Unlike Claim~1, where we construct test sections coming from fixed vectors in $\mathbb{C}^{n+1}$, now we are going to obtain test sections coming from fixed matrices $B \in HM(n+1)$, the space of Hermitian matrices of order $n+1$ (see Section~\ref{ap:complex-projective-space}). To do so,  we consider the Berger sphere $\mathbb{S}^{2n+1}_\tau$ isometrically embedded in $\mathbb{CP}^{n+1}(4(1-\tau^2))$ (see Proposition~\ref{prop:geodesic-spheres-complex-spaces}) and the complex projective space isometrically embedded in $HM(n+1)$ (see Proposition~\ref{prop:properties-embedding-complex-projective-euclidean}). 

In fact, given $B \in HM(n+1)$, from \autoref{prop:properties-embedding-complex-projective-euclidean}, it is not difficult to check that its tangent component $X_B$ to $\mathbb{CP}^{n+1}(4(1-\tau^2))$ is a holomorphic vector field. So, we consider the orthogonal decomposition 
\[
  B = X_B + B^N =  B^\top + B^\perp + \bprodesc{B}{J\xi}J\xi + B^N,
\]
where $B^N$ denotes the normal component of $B$ to $T:\mathbb{CP}^{n+1}(4(1-\tau^2)) \to HM(n+1)$ (see \autoref{prop:properties-embedding-complex-projective-euclidean}), $B^\top$ is the tangent component to $M$, $B^\perp$ is the normal component to $\Phi: M \to \mathbb{S}^{2n+1}_\tau$, and $\bprodesc{B}{J\xi}J\xi$ the normal component to $F:\mathbb{S}^{2n+1}_\tau \to \mathbb{CP}^{n+1}(4(1-\tau^2))$ (see \autoref{prop:geodesic-spheres-complex-spaces}).

Now, for each $B \in HM(n+1)$, we consider the test normal section $B^\perp$ and our goal is to compute the quadratic form $\mathcal{Q}(B^\perp)$. Let $\{e_i\colon i = 1,\ldots,d\}$ be a orthonormal tangent reference to $M$. We will make the computation at a point $p \in M$ assuming that $(\nabla_{e_i}e_j)_p = 0$.

Firstly, since the covariant derivative of $B$ in $\HM(n+1)$ vanishes, we deduce that, for any $u\in T_pM$,
\begin{align*}
	\nabla_u B^\top &= A_{B^\perp} u - \bprodesc{B}{J\xi} (\nablap_u J\xi)^\top  + (\overline{A}_{B^N}u)^\top,\label{eq:derivative-B-tangent}\\
	\nabla^\perp_u B^\perp &= -\sigma(u, B^\top) - \bprodesc{B}{J\xi} (\nablap_u J\xi)^\perp + (\overline{A}_{B^N}u)^\perp,
\end{align*}
where $\nablap$ is the Levi-Civita connection of $\mathbb{CP}^{n+1}(4(1-\tau^2))$ and $\overline{A}$ is the shape operator of $\mathbb{CP}^{n+1}(4(1-\tau^2))$ in $\HM(n+1)$. Now, using that $\xi$ is tangent to $M$, \eqref{eq:2ff-Berger-complex-space} and \eqref{eq:covariant-derivative-xi} we get
\[
  \nablap_u J\xi = J\nablap_u \xi = J \left( \nablab_u \xi + \hat{\sigma}(u,\xi) \right) = -\tau u + \tfrac{1-\tau^2}{\tau} \bprodesc{u}{\xi}\xi.
\] 
Hence, the previous two equalities become in
\begin{align}
	\nabla_u B^\top &= A_{B^\perp} u + \bprodesc{B}{J\xi}\left( \tau u - \tfrac{1-\tau^2}{\tau} \bprodesc{u}{\xi}\xi \right)  + (\overline{A}_{B^N}u)^\top,\label{eq:derivative-B-tangent}\\
	\nabla^\perp_u B^\perp &= -\sigma(u, B^\top) + (\overline{A}_{B^N}u)^\perp.\label{eq:derivative-B-normal}
\end{align}

Now, using Codazzi equation, the minimality of $\Phi$, \eqref{eq:derivative-B-tangent} and~\eqref{eq:derivative-B-normal}, and the fact that $\sigma(e_i,\xi) = \tau (Je_i)^\perp$ thanks to \eqref{eq:covariant-derivative-xi}, we get
\[
\Delta^\perp B^\perp = -\mathcal{A}B^\perp + \sum_{i = 1}^{d} \left[ (\bR(B^\top, e_i)e_i)^\perp - \sigma(e_i, (\overline{A}_{B^N}e_i)^\top)+ \nabla^\perp_{e_i} (\overline{A}_{B^N}e_i)^\perp \right].
\] 

Then, taking into account that $\xi$ is tangent to $M$, that $B^\top + B^\perp = X_B - \bprodesc{B}{J\xi}J\xi$, using \eqref{eq:Riemann-tensor-Berger}, \eqref{eq:Jacobi-operator} and the previous formula, we deduce that
\begin{multline}\label{eq:Jacobi-operator-B-normal}
    \mathcal{L} B^\perp = [d - (1-\tau^2)]B^\perp - 3(1-\tau^2)[J(JX_B)^\top]^\perp \\
                        - \sum_{i = 1}^{d} \left[\sigma(e_i, (\overline{A}_{B^N}e_i)^\top) - \nabla^\perp_{e_i} (\overline{A}_{B^N}e_i)^\perp\right].
\end{multline} 

Let $\{B_j\colon j = 1,\ldots, (n+1)^2\}$ be a orthonormal reference in $\HM(n+1)$. We compute the sum $\sum_j \bprodesc{\mathcal{L} B_j^\perp}{B_j^\perp}$. Notice that
\[
\begin{split}
  -\sum_{j} \bprodesc{J(JX_{B_j})^\top}{B_j^\perp} &= \sum_{i,j} \bprodesc{(JX_{B_j})^\top}{JB_j^\perp} = \sum_{i,j} \bprodesc{JX_{B_j}}{e_i} \bprodesc{e_i}{JB_j^\perp} \\  
                                                   &= \sum_{i,j} \bprodesc{X_{B_j}}{Je_i} \bprodesc{B_j}{(Je_i)^\perp} = \sum_{i,j} \bprodesc{B_j - B_j^N}{Je_i} \bprodesc{B_j}{(Je_i)^\perp}  \\
                                                   &= \sum_{i = 1}^d \lvert(Je_i)^\perp\rvert^2,\\
	\sum_{j} \lvert B_j^\perp\rvert^2 &= \sum_{j, \alpha} \bprodesc{B_j}{\zeta_\alpha}^2 = \sum_{\alpha = 1}^{q} \lvert\zeta_\alpha\rvert^2 = q, \\
	\sum_{j} \bprodesc{\sigma(e_i, (\overline{A}_{B_j^N}e_i)^\top)}{B_j^\perp} &= \sum_{j, k} \bprodesc{\overline{A}_{B_j^N}e_i}{e_k} \bprodesc{\sigma(e_i, e_k)}{B_j} = \sum_{i, k} \bprodesc{\overline{\sigma}(e_i, e_k)}{\sigma(e_i, e_k)} = 0,
\end{split}
\] 
where $\{\zeta_\alpha\colon \alpha = 1,\ldots, q\}$ represents an orthonormal reference of the normal bundle of $\Phi$. 

We now compute the last term in \eqref{eq:Jacobi-operator-B-normal}
\[
\begin{split}
	\sum_{i,j} \bprodesc{\nabla^\perp_{e_i} (\overline{A}_{B_j^N}e_i)^\perp}{B_j^\perp} &= \sum_{i,j,\alpha} \bprodesc{\nabla^\perp_{e_i} (\bprodesc{\overline{A}_{B_j^N}e_i}{\zeta_\alpha}\zeta_\alpha)}{B_j^\perp} = \sum_{i,j,\alpha} \bprodesc{\nabla^\perp_{e_i} \bprodesc{\overline{\sigma}(e_i,\zeta_\alpha)}{B_j} \zeta_\alpha}{B_j}  \\
	&= \sum_{i,j,\alpha} \left[ e_i(\bprodesc{\overline{\sigma}(e_i,\zeta_\alpha)}{B_j}) \bprodesc{\zeta_\alpha}{B_j} + \bprodesc{\overline{\sigma}(e_i, \zeta_\alpha)}{B_j} \bprodesc{\nabla^\perp_{e_i} \zeta_\alpha}{B_j} \right] \\%
	&= \sum_{i,j,\alpha} \bprodesc{\zeta_\alpha}{B_j} \bprodesc{-\overline{A}_{\overline{\sigma}(e_i,\zeta_\alpha)}e_i + \nablap^N_{e_i} \overline{\sigma}(e_i,\zeta_\alpha)}{B_j} \\
	&= -\sum_{i,\alpha} \bprodesc{\overline{A}_{\overline{\sigma}(e_i,\zeta_\alpha)}e_i}{\zeta_\alpha} + \sum_{i,\alpha} \bprodesc{\nablap^N_{e_i} \overline{\sigma}(e_i,\zeta_\alpha)}{\zeta_\alpha} \\
  &= - \sum_{i,\alpha} \lvert\overline{\sigma}(e_i, \zeta_\alpha)\rvert^2 = -(1-\tau^2)dq + (1-\tau^2) \sum_{i=1}^{d} \lvert(Je_i)^\perp\rvert^2,
\end{split}
\] 
where we have used that $\{B_j\}$ is an orthonormal reference in $\HM(n+1)$,  that $\overline{\sigma}$ is parallel (see \autoref{prop:properties-embedding-complex-projective-euclidean}), and \eqref{eq:2ff-complex-projective-euclidean}.
We finally obtain that
\[
\begin{split}
  \sum_{j = 1}^{(n+1)^2}\mathcal{Q}(B_j^{\perp}) &= - \sum_{j = 1}^{(n+1)^2} \int_M\bprodesc{\mathcal{L} B_j^\perp}{B_j^\perp}\,dv\\
                                                 &= -\int_Mq[\tau^2(d+1)-1]\,\mathrm{d}v -  4(1-\tau^2)  \sum_{i = 1}^{d} \int_M\lvert(Je_i)^\perp\rvert^2 \,\mathrm{d}v\\
  &=-  \frac{4d}{d+1}  \sum_{i = 1}^{d} \int_M\lvert(Je_i)^\perp\rvert^2\, \mathrm{d}v,
\end{split}
\] 
where we have used the hypothesis $\tau^2 = \frac{1}{d+1}$ in the last equality. As a consequence, if  $\Phi$ is stable we obtain that $(Je_i)^\perp = 0$ for any $i \in \{1,\ldots, d\}$. This proves Claim 2.

Finally, the results follows from~\autoref{prop:killing-tangente-normal}.(i) 
\end{proof}

Using some ideas developed in the proof of \autoref{thm:classification-stable}, we can estimate the index of the  minimal submanifolds of $\mathbb{S}^{2n+1}_{\tau}$ with either $\xi^{\perp}=0$ or $\xi^{\top}=0$.
\begin{corollary}\label{co:properties-immersion-xi-tangent-or-normal}
	Let $\Phi: M^d \to \mathbb{S}^{2n+1}_\tau$ be a minimal immersion of a compact manifold $M$. Then:
	\begin{enumerate}[(i)]
		\item If $\xi^\perp = 0$,  then $ \tfrac{1}{\tau^2}-(d+1)$ is an eigenvalue of the Jacobi operator $\mathcal{L}$. Its  multiplicity is either $2n+1-d$ (and $\Phi(M)$ is congruent to the totally geodesic Berger sphere $\mathbb{S}^{2m+1}_\tau\subset\mathbb{S}^{2n+1}_\tau$) or $\geq 2(n+1)$.
		
		Moreover, if $\frac{1}{d+1} < \tau^2 \leq 1$ then 
		\begin{enumerate}
			\item $\Ind(M) \geq 2n+1-d$ and the equality is attained if and only if $d = 2m+1$ and $M$ is congruent to the totally geodesic Berger sphere $\mathbb{S}^{2m+1}_\tau\subset\mathbb{S}^{2n+1}_\tau$.
			\item If $M$ is not the totally geodesic Berger sphere then $\Ind(M) \geq 2(n+1)$.
			\item If $M$ is an orientable hypersurface  then $\Ind(M) \geq 2n+3$.
		\end{enumerate}
		
		\item If $\xi^\top = 0$,  then $\Ind(\Phi) \geq 2(n+1)$.
	\end{enumerate}
	
\end{corollary}

\begin{proof}
Following the proof of Claim 1 in \autoref{thm:classification-stable}, for each $a \in \mathbb{C}^{n+1}$ we have computed in \eqref{eq:Jacobi-Xa-normal} $\mathcal{L}X_a^{\perp}$, where $X_a$ is the vector field on $\mathbb{S}^{2n+1}_{\tau}$ defined by  $X_a = a - \prodesc{a}{\Phi}\Phi$. 

	(i) Since $\xi^\perp = 0$,  using \eqref{eq:covariant-derivative-xi}, we deduce that $\sigma(\xi, u) = \tau (Ju)^\perp$ for any tangent vector  $u\in TM$. As a consequence, from \eqref{eq:Jacobi-Xa-normal} we get 
	\[
	\mathcal{L} X_a^\perp + \bigl(\tfrac{1}{\tau^2}-(d + 1)\bigr)X_a^\perp = 0.
	\] 
	So the normal sections of $W=\{X_a^{\perp}\colon a\in \mathbb{C}^{n+1}\}$ are eigensections of the eigenvalue $\frac{1}{\tau^2}-(d+1)$.
	
  If $\dim W < 2(n+1)$, then there exists a nonzero vector $a\in\mathbb{C}^{n+1}$ such that $X_a^{\perp}=0$. Hence $X_a$ is tangent to $M$. Now, since $\xi^\perp = 0$ we know that $T^{\perp}M=T^{\perp}_gM$, where $T^\perp_g M$ stands for the normal bundle of the immersion $\Phi: M \to \mathbb{S}^{2n+1}$. Then, it is not difficult to check that Hessian with respect to the metric $g$ of the function $h=g(\Phi, a)$ satisfies $\hbox{Hess}\, h=-hg$, and so using Obata's theorem~\cite{Obata1962}, we have that $(M,g)$ is isometric to a unit sphere or $h=0$. In the first case, from \autoref{prop:killing-tangente-normal}, $\Phi:M\rightarrow \mathbb{S}^{2n+1}$ is also a minimal immersion and so, from the Gauss equation, $\Phi:M\rightarrow\mathbb{S}^{2n+1}$ is totally geodesic and $\dim W=2(n+1)-d$. Now, using \eqref{eq:Levi-Civita-connection-Berger-round} and \eqref{eq:covariant-derivative-xi} we get that $\Phi:M\rightarrow \mathbb{S}^{2n+1}_{\tau}$ is also totally geodesic and from \autoref{prop:killing-tangente-normal}, $M$ is congruent to the totally geodesic Berger sphere $\mathbb{S}^{2m+1}_{\tau}\subset\mathbb{S}^{2n+1}_{\tau}$. In the second case, $h=0$,  its gradient $\nablag h = X_a$ also vanishes which is imposible since $a \neq 0$. 

  On the other hand, the other possibility is that $\dim W\geq 2(n+1)$.

  Now, as consequence of this and \autoref{prop:index-nullity-totally-geodesic-Berger-spheres}, (a) and (b) follow easily. If $M$ is an orientable hypersurface, then the Jacobi operator is a Schrödinger operator acting on $\mathcal{C}^{\infty}(M)$. It is well-known that its first eigenvalue has multiplicity $1$, and hence
  $\frac{1}{\tau^2}-(d + 1)$ can not be  the first eigenvalue of the Jacobi operator.  Hence $\Ind (M)\geq 1+2(n+1)$ which proves (c).

  (ii) Assume that $\xi^\top = 0$. Then, from \autoref{prop:classification-totally-geodesic}.(ii) and~\eqref{eq:Jacobi-Xa-normal} we get 
  \[
    \mathcal{L} X_a^\perp = d X_a^\perp -2(1-\tau^2)[J(JX_a)^\top]^\perp,
  \] 
  for any vector $a\in\mathbb{C}^{n+1}$. But from \autoref{prop:killing-tangente-normal} the normal bundle of $\Phi$ decomposes orthogonally as
  \[
    T^{\perp}M=J(TM)\oplus \langle\xi\rangle \oplus \mathcal{D},
  \]
  where $\mathcal{D}$ is a subbundle invariant by $J$. It is clear that $[J(JX_a)^\top]^\perp=J(JX_a)^\top$ is a normal section on $J(TM)$. Moreover, if $Y$ is any vector field tangent to $M$, then $\langle J(JX_a)^\top, JY\rangle=\langle JX_a,Y\rangle=-\langle X_a,JY\rangle$ and so $-J(JX_a)^\top$ is 
  the component $X_a^{J(TM)}$ of the vector field $X_a$ in the subbundle $J(TM)$. Hence, for any $a\in\mathbb{C}^{n+1}$ we have
  \[
    \mathcal{L}X_a^\perp=dX_a^\perp+2(1-\tau^2)X_a^{J(TM)}.
  \]
  So
  \[
    \mathcal{Q}(X_a^\perp)=-d\int_M \lvert X_a^\perp\rvert^2\, \mathrm{d}v-2(1-\tau^2)\int_M\lvert X_a^{J(TM)}\rvert^2\, \mathrm{d}v\leq -d\int_M\lvert X_a^\perp\rvert^2\, \mathrm{d}v.
  \]
  From this inequality we have that $\Ind (M)\geq \dim \{X_a^{\perp}\colon  a\in\mathbb{C}^{n+1}\}$. Now following a similar reasoning as in (i) we conclude that either $M$ is the totally geodesic sphere $\mathbb{S}^d$ given in \autoref{prop:killing-tangente-normal} or $\Ind(M)\geq 2(n+1)$. But, by \autoref{prop:index-nullity-totally-geodesic-spheres}, $\Ind(\mathbb{S}^d)\geq 2(n+1)$ and so the proof finishes.
\end{proof}

\section{The case of surfaces in $\mathbb{S}^3_{\tau}$}\label{sec:compact-stable-minimal-surfaces}
Let $\Phi:M\rightarrow \mathbb{S}^3_\tau$  be a  minimal immersion of a compact orientable surface $M$. Then $\Phi$ is two-sided and, if $N$ is  a unit normal vector field to $\Phi$, we choose the orientation on $M$ which makes $\{\xi^\top, (JN)^\top\}$  a positively oriented reference on the open set $\{p\in M\colon \nu^2(p)<1\}$, where $\nu = \bprodesc{N}{\xi}$.  If $z=x+iy$ is a conformal parameter on $M$, then the two-differential 
\[
  \Theta(z)=\Big(\langle\sigma(\partial_z,\partial_z),N\rangle+\tfrac{2i(1-\tau^2)}{\tau}\langle\Phi_z,\xi\rangle^2\Big)\mathrm{d}z\otimes \mathrm{d}z
\]
is holomorphic \cite{AR05, FM07}, where $\partial_z=\frac{1}{2}(\partial_x-i\partial_y)$.

If $M$ is a sphere, then $\Theta=0$ and so, from last equation and the above remark about the orientation, it follows that
\[
  \sigma(X,Y)=\tfrac{1-\tau^2}{\tau}\{ \langle Y,\xi\rangle\langle X,JN\rangle+\langle X,\xi\rangle\langle Y,JN\rangle \},
\]
for any tangent vector fields $X,Y$. So, \eqref{eq:Levi-Civita-connection-Berger-round} becomes in
\[
  \nablar_X Y=	\nabla_X Y + \tfrac{1 - \tau^2}{\tau}\bigl[ \bprodesc{Y}{\xi} (JX)^\top + \bprodesc{X}{\xi} (JY)^\top \bigr]\in \mathfrak{X}(M),
\] 
which says that $\sigma^g=0$. So, $\Phi:M\rightarrow\mathbb{S}^3$ is totally geodesic and hence, up to congruences, there exists only one minimal sphere in $\mathbb{S}^3_{\tau}$, which is given by
\[
  S^2=\{(z_1,z_2)\in \mathbb{S}^3_{\tau}\colon  \pIm z_2=0\}.
\]

On the other hand, there is only a Clifford surface in $\mathbb{S}^3_{\tau}$ (see \autoref{ex:Clifford-hypersurfaces}) which is given by
\[
  T^2_{\tau}=\mathbb{S}^1\left(\tfrac{1}{\sqrt 2}\right)\times \mathbb{S}^1\left(\tfrac{1}{\sqrt 2}\right)=\left\{(z_1,z_2)\in\mathbb{C}^2\colon \lvert z_i\lvert=\tfrac{1}{\sqrt{2}}\right\}\subset \mathbb{S}^3_{\tau}.
\]
As $\lvert\sigma\rvert^2=2\tau^2$ (see \autoref{ex:Clifford-hypersurfaces}) and $\nu=0$, from the Gauss equation it follows that $T^2_{\tau}$ is flat. Also, from \cite{TU2010} we have that the finite coverings of the  Clifford surface are the only minimal flat tori in $\mathbb{S}^3_{\tau}$. Hence we have a parameter family $\{T^2_{\tau}\colon 0<\tau\leq 1\}$ of flat tori, whose conformal structures we are going to determine.

The universal covering of $T^2_{\tau}$ is  $(t,s)\in\mathbb{R}^2\mapsto\frac{1}{\sqrt{2}}(e^{it},e^{is})$ and so, as $\xi^{\perp}=0$,  from ~\eqref{eq:defini-metrica Berger} we obtain 
\[
  \langle\partial t,\partial t\rangle=\langle\partial s,\partial s\rangle=\tfrac{1}{4}(1+\tau^2),\quad \langle\partial t,\partial s\rangle=\tfrac{1}{4}(\tau^2-1). 
\] 
Hence $T^2_{\tau}=\mathbb{R}^2/\Lambda_{\tau}$ where $\Lambda_{\tau}$ is the lattice of $\mathbb{R}^2$ generated by $\{\pi(\tau,1),\pi(\tau,-1)\}$. Using the following representation of the conformal structures of the tori (see \cite{BGM71} and \autoref{fig:toros}), 
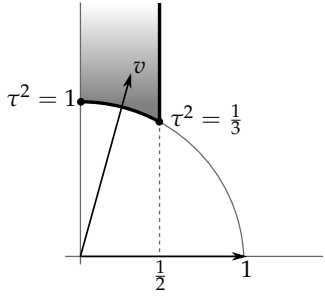
\begin{figure}[htbp]
  \small
  \centering
  \def\svgwidth{4cm}
  \begingroup%
    \makeatletter%
    \providecommand\color[2][]{%
      \errmessage{(Inkscape) Color is used for the text in Inkscape, but the package 'color.sty' is not loaded}%
      \renewcommand\color[2][]{}%
    }%
    \providecommand\transparent[1]{%
      \errmessage{(Inkscape) Transparency is used (non-zero) for the text in Inkscape, but the package 'transparent.sty' is not loaded}%
      \renewcommand\transparent[1]{}%
    }%
    \providecommand\rotatebox[2]{#2}%
    \newcommand*\fsize{\dimexpr\f@size pt\relax}%
    \newcommand*\lineheight[1]{\fontsize{\fsize}{#1\fsize}\selectfont}%
    \ifx\svgwidth\undefined%
      \setlength{\unitlength}{180.52861972bp}%
      \ifx\svgscale\undefined%
        \relax%
      \else%
        \setlength{\unitlength}{\unitlength * \real{\svgscale}}%
      \fi%
    \else%
      \setlength{\unitlength}{\svgwidth}%
    \fi%
    \global\let\svgwidth\undefined%
    \global\let\svgscale\undefined%
    \makeatother%
    \begin{picture}(1,0.91626891)%
      \lineheight{1}%
      \setlength\tabcolsep{0pt}%
      \put(0,0){\includegraphics[width=\unitlength,page=1]{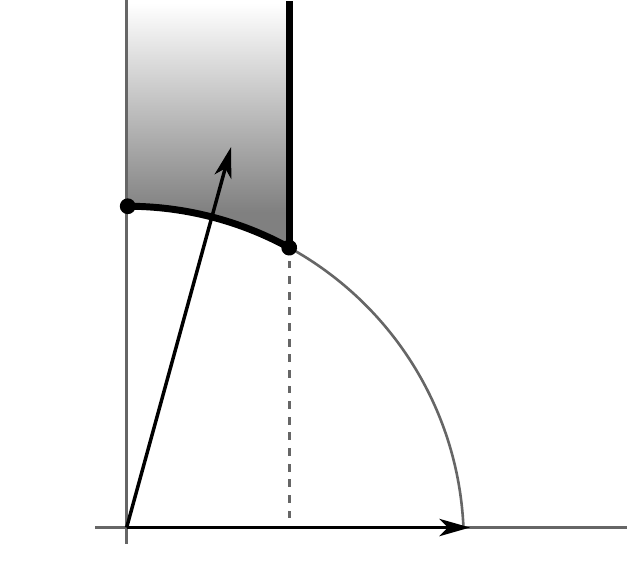}}%
      \put(0.49725927,0.50634717){\color[rgb]{0,0,0}\makebox(0,0)[lt]{\lineheight{1.25}\smash{\begin{tabular}[t]{l}$\tau^2=\frac13$\end{tabular}}}}%
      \put(-0.04074406,0.57489586){\color[rgb]{0,0,0}\makebox(0,0)[lt]{\lineheight{1.25}\smash{\begin{tabular}[t]{l}$\tau^2=1$\end{tabular}}}}%
      \put(0.72990944,0.00781123){\color[rgb]{0,0,0}\makebox(0,0)[lt]{\lineheight{1.25}\smash{\begin{tabular}[t]{l}$1$\end{tabular}}}}%
      \put(0.43909678,-0.0108838){\color[rgb]{0,0,0}\makebox(0,0)[lt]{\lineheight{1.25}\smash{\begin{tabular}[t]{l}$\frac12$\end{tabular}}}}%
      \put(0.37597917,0.68154158){\color[rgb]{0,0,0}\makebox(0,0)[lt]{\lineheight{1.25}\smash{\begin{tabular}[t]{l}$v$\end{tabular}}}}%
    \end{picture}%
  \endgroup%
  \caption{Each vector $v$ in the shaded region represents a conformal structure on a torus (generated by the lattice $\{(1,0), v\}$).}
  \label{fig:toros}
\end{figure}
we have that the torus $T^2_{\tau}$ is conformally equivalent to the torus generated by the lattice $\{(1,0), (\frac{1-\tau^2}{1+\tau^2},\frac{2\tau}{1+\tau^2})\}$, for $\tau^2 \geq \frac{1}{3}$, and it is conformally equivalent to the torus generated by the lattice $\{(1,0), (\frac{1}{2}, \frac{1}{2\tau})\}$ when $\tau^2 \leq \frac{1}{3}$. This means that the conformal structures of $\{T^2_{\tau}\colon  0<\tau^2\leq 1\}$ sweep the thick curve in \autoref{fig:toros}. So the conformal structure of $T^2_{1/\sqrt{3}}$ is the equilateral one.

The next result can be considered as the version for the Berger sphere $\mathbb{S}^3_{\tau}$ of the result in \cite{Urbano1990}, where the Clifford torus in $\mathbb{S}^3$  was characterized by its index. We remark that in the theorem   $\frac{1}{\hbox{dim}(M)+1}\leq\tau^2\leq 1$ like in the previous results.
\begin{theorem}\label{thm:characterization-index-surfaces}
	Let $\Phi:M\rightarrow \mathbb{S}^3_\tau$  be a  minimal immersion   of a compact orientable surface $M$ with $\frac{1}{3} \leq \tau^2 \leq 1$. Then 
  \begin{enumerate}[(i)]
    \item $\Ind(M) = 1$ if and only if  $\Phi$ is an embedding and $M$ is either the minimal sphere $S^2$ or $\tau^2 = \frac{1}{3}$ and $M$ is the Clifford surface ${T}^2_{1/\sqrt 3}$.
	\item $\Ind(M)\geq \frac{g}{4}$, where $g$ is the genus of $M$.
	\end{enumerate}
\end{theorem}

\begin{proof}
(i) From \autoref{prop:Clifford-hypersurfaces-index-nullity}, the Clifford surface $T^2_{1/\sqrt  3}$  has  index one. Also, in ~\cite[\S4]{TU2011} it was  proved that the index of the minimal sphere $S^2$ is also one for any $\tau\in (0,1].$ 

Suppose now that $\Ind(M)=1$. If $g = 0$, then $M$ is the minimal sphere $S^2$. Hence we can assume that the genus of $M$ is $g\geq 1$.

Thanks to \autoref{prop:properties-embedding-complex-projective-euclidean}, we can consider the  isometric embeddings
\[
  \mathbb{CP}^2(4(1-\tau^2))\subset\mathbb{S}^7(c,r)\subset HM_1(3)\subset H(3),
\] 
where $\mathbb{S}^7(c,r)$ stands for the sphere of center $c=(\frac{1}{3\sqrt{2(1-\tau^2)}}) I$ and radius $r = \frac{1}{\sqrt{3(1-\tau^2)}}$, and $HM_1(3)$ is the affine hyperplane of the Hermitian matrices of order three $H(3)$ defined in \eqref{eq:hermitian-matrices-linear-subespace-HM1}.  We then consider the immersion $\Psi = \frac{1}{r}(\Phi-c): M \to \mathbb{S}^7(0, 1)$.

As $M$ has index $1$,  let $\varphi$ be an eigenfunction associated to the first eigenvalue $\lambda_1<0$ of the Jacobi operator $\mathcal{L}$. Then, by \cite{LY82} there exists $B \in H_1(3)$, $\lvert B\rvert< 1$, and a conformal transformation $F_B: \mathbb{S}^7(0, 1) \to \mathbb{S}^7(0,1)$, $$F_B(p) = B + \tfrac{1-\lvert B\rvert^2}{\lvert p + B\rvert^2}(p+B)\quad\quad
 \forall p\in\mathbb{S}^7(0,1),$$ such that
\[
  \int_M  \varphi\cdot(F_B\circ \Psi)\,\mathrm{d}v  = 0.
\] 
As the second eigenvalue of $\mathcal{L}$ is non-negative, the above expression  implies that  $\mathcal{Q}(F_B \circ \Psi) = \sum_{i = 1}^8 \mathcal{Q}((F_B\circ \Psi)_i)\geq 0$. 

Now, if $K$ is the Gauss curvature of $\Phi$,  the Gauss equation $K=-\frac{\lvert\sigma\rvert^2}{2}+\tau^2+4(1-\tau^2)\nu^2$ of $\Phi$ joint with \eqref{eq:Jacobi-operator-hypersurfaces}, allow to write the Jacobi operator of $\Phi$ as
\begin{equation}\label{eq:Jacobi-operator-surface}
\mathcal{L}=\Delta-2K+4+4(1-\tau^2)\nu^2.
\end{equation}
So, using that $\lvert F_B\circ\Psi\rvert^2=1$, we obtain
\[
  \mathcal{Q}(F_B \circ \Psi) = \int_M \lvert \nabla(F_B \circ \Psi)\rvert^2 + \bigl[ 2K - 4 - 4(1-\tau^2)\nu^2 \bigr]\,\mathrm{d}v \geq 0,
\] 
which can be rewritten, using Gauss-Bonnet theorem, as
\begin{equation}\label{eq:inequality-gradient-FaPsi}
  \int_M \lvert \nabla(F_B \circ \Psi)\rvert^2\,\mathrm{d}v \geq 8\pi(g-1) + 4\int_M [1 + (1-\tau^2)\nu^2]\,\mathrm{d}v.
\end{equation} 
As $F_B$ is a conformal transformation of the sphere and $\mathrm{d}\Psi=\frac{1}{r}\mathrm{d}\Phi$ we easily get
\[
  \lvert \nabla(F_B \circ \Psi)\rvert^2 = \frac{2}{r^2} \frac{(1-\lvert B\rvert^2)^2}{\lvert\Psi + B\rvert^4},
\]
 and so, the inequality~\eqref{eq:inequality-gradient-FaPsi} becomes in
\begin{equation}\label{eq:inequality1}
 \int_M\frac{(1-\lvert B\rvert^2)^2}{r^2\lvert\Psi + B\rvert^4}\, \mathrm{d}v\geq 4\pi(g-1)+2\int_M [1 + (1-\tau^2)\nu^2]\,\mathrm{d}v.
\end{equation} 
To estimate the first term of \eqref{eq:inequality1}, we are going to compute $\Delta\log\lvert\Psi+B\rvert^2$, which is a well defined function because $\lvert B\rvert<1$. To do it, we decompose 
\[
B = B^\top + \bprodesc{B}{N}N + B^\perp + \bprodesc{\Psi}{B}\Psi,
\]
where $B^\top$ is the tangent component of $B$ to $M$ and $B^{\perp}$ is the normal component of $B$ to $\mathbb{S}^3_{\tau}\subset S^7(c,r)$. From this equation it is clear that $\nabla\lvert\Psi+B\rvert^2=\frac{2}{r}B^\top$, and so 
\[
\Delta\lvert\Psi+B\rvert^2=\frac{2}{r}\sum_{i=1}^2\langle\widetilde{\sigma}(e_i,e_i),B\rangle -\frac{4}{r^2}\langle\Psi,B\rangle,
\]
where $\widetilde{\sigma}$ is the second fundamental form of the embedding $\mathbb{S}^3_{\tau}\subset \mathbb{S}^7(c,r)$. Hence we obtain that
\[
\Delta\log\lvert\Psi+B\rvert^2=\frac{2}{r\lvert\Psi+B\rvert^2}\sum_{i=1}^2\langle\widetilde{\sigma}(e_i,e_i),B\rangle-\frac{4\langle\Psi,B\rangle}{r^2\lvert\Psi+B\rvert^2}-\frac{4\lvert B^\top\rvert^2}{r^2\lvert\Psi+B\rvert^4}.
\]
 Now, 
\begin{equation}\label{eq:*}
\begin{aligned}
\int_M\frac{(1-\lvert B\rvert^2)^2}{r^2\lvert\Psi + B\rvert^4}\mathrm{d}v=\int_M\left[\frac{1}{r^2}-\frac{4}{r^2\lvert \Psi+B\rvert^4}(\lvert B\rvert^2+\langle\Psi,B\rangle^2+\langle\Psi,B\rangle(1+\lvert B\rvert^2))\right]\,\mathrm{d}v\\
\leq \int_M \left[\frac{1}{r^2}-\frac{4}{r^2\lvert\Psi+B\rvert^4}(\lvert B^\top\rvert^2+\lvert B^{\perp}\rvert^2+2\langle\Psi,B\rangle^2+\langle\Psi,B\rangle(1+\lvert B\rvert^2))\right]\,\mathrm{d}v\\
=\int_M\left[\frac{1}{r^2}-\frac{4\lvert B^\perp\rvert^2}{r^2\lvert\Psi+B\rvert^4}-\frac{2}{r\lvert\Psi+B\rvert^2}\sum_{i=1}^2\bprodesc{\widetilde{\sigma}(e_i,e_i)}{B}\right]\,
 \mathrm{d}v,
\end{aligned}
\end{equation}
where the first equality is a direct computation, the second inequality comes from $\lvert B\rvert^2\geq \lvert B^\top\rvert^2+\lvert B^\perp\rvert^2+\langle\Psi,B\rangle^2$ and the third equality comes from $\int_M\Delta\log\lvert\Psi+B\rvert^2\, \mathrm{d}v=0$. 

On the other hand, using the inequality
\[
      0 \leq \Bigl\lvert \frac{B^\perp}{r\lvert\Psi + B\rvert^2} + \tfrac{1}{4} \sum_{i=1}^2\widetilde{\sigma}(e_i,e_i) \Bigr\rvert^2 = \frac{\lvert B^\perp\rvert^2}{r^2 \lvert\Psi+B\rvert^4} + \frac{\sum_{i=1}^2\langle\widetilde{\sigma}(e_i,e_i),B\rangle }{2r \lvert\Psi + B\rvert^2}   
     + \frac{1}{16} \Bigr\lvert \sum_{i=1}^2\widetilde{\sigma}(e_i,e_i) \Bigr\rvert^2,
  \] 
 in \eqref{eq:*} we get that
\[
\int_M\frac{(1-\lvert B\rvert^2)^2}{r^2\lvert\Psi + B\rvert^4}\mathrm{d}v\leq \int_M\left[\frac{1}{r^2}+\frac{1}{4} \Bigl\lvert \sum_{i=1}^2\widetilde{\sigma}(e_i,e_i) \Bigl\rvert^2\right]\,\mathrm{d}v.
\]
So, using that $1/r^2=3(1-\tau^2)$, the previous inequality and~\eqref{eq:inequality1} becomes in
\[
  \tfrac{1}{4}\int_M \Bigl\lvert\sum_{i=1}^2\widetilde{\sigma}(e_i,e_i)\Bigr\rvert^2\, \mathrm{d}v\geq 4\pi(g-1)+\int_M \bigl[3\tau^2-1+2(1-\tau^2)\nu^2\bigr] \,\mathrm{d}v.
\] 
But the second fundamental form $\widetilde{\sigma}$ of the embedding $\mathbb{S}^3_{\tau}\subset\mathbb{S}^7(c,r)$ is given $\widetilde{\sigma}=\hat{\sigma}+\overline{\sigma}+\frac{1}{r}\bprodesc{\cdot}{\cdot}\Psi$, where $\overline{\sigma}$ is the second fundamental form of the embedding $\mathbb{CP}^2(4(1-\tau^2))\subset HM(3)$. Now, from \eqref{eq:2ff-Berger-complex-space} and  \eqref{eq:2ff-complex-projective-euclidean}, we easily get that 
\[
  \tfrac{1}{4}\Bigl\lvert\sum_{i=1}^2\widetilde{\sigma}(e_i,e_i)\Bigr\rvert^2= \tfrac{1}{4\tau^2} \bigl[3\tau^2 - 1 + (1-\tau^2)\nu^2\bigr]^2 +(1-\tau^2)\nu^2.
\] 
Finally, the last inequality above reads
\[
  4\pi(g-1) + \tfrac{1}{4\tau^2}\int_M \bigl[3\tau^2 - 1 + (1-\tau^2)\nu^2\bigr] \cdot \bigl[\tau^2(1+\nu^2) + (1-\nu^2)  \bigr]\, \mathrm{d}v  \leq 0
\] 
Since $\frac{1}{3} \leq \tau^2 \leq 1$, $g \geq 1$ and $\nu^2 \leq 1$ we get that $g = 1$, $\tau^2 = \frac{1}{3}$ and $\nu = 0$. Therefore, the Killing field $\xi$ is tangent to $M$ and so an orthonormal reference on $TM$ is given by $\{\xi, JN\}$, where $N$ is a unit normal vector field to $\Phi$. Now, from \eqref{eq:covariant-derivative-xi} it follows that
\[
\sigma(\xi,\xi)=0, \quad \sigma(JN,\xi)=-\tau N,
\]
which implies that $\lvert\sigma\rvert^2=2\tau^2$. The Gauss equation says us that $M$ is flat and so $M$ is congruent to a finite covering of the Clifford torus (\cite{TU2010}). As the Jacobi operator is $\mathcal{L}f=\Delta+4$, $M$ is congruent to the Clifford torus.

(ii) The argument we use to prove (ii) is inspired in the papers \cite{Ros06, Savo2010}.

As $\Ind(M)\geq 1$, we can assume  $g\geq 5$. Then, if $\hat\Delta$ is the Hodge Laplacian acting on $1$-forms on $M$, it is well-known that $\ker \hat\Delta$ is the space of harmonic 1-forms of M, whose dimension is $2g$. Using the metric on $M$, the $1$-forms and the vector fields on $M$ are identified and we say that a vector field $X$ is harmonic if the corresponding $1$-form is harmonic. It is well-known that this property is equivalent to
\[
\diver X=0,\quad \langle\nabla_vX,w\rangle=\langle\nabla_wX,v\rangle,
\]
for any tangent vectors $v,w$ to $M$. Also, if $X$ is a harmonic vector field on $M$ then
\[
\Delta X=KX,
\]
where $K$ is the Gauss curvature of $M$ and $\Delta$ is the rough Laplacian defined by $\Delta=\sum_{i=1}^2\{\nabla_{e_i}\nabla_{e_i}-\nabla_{\nabla_{e_i}e_i}\}$.

Given a harmonic vector field $X$ on $M$ we consider the vectorial function $X:M\rightarrow HM_1(3)$ and we are going to compute $\mathcal{Q}(X)=\sum_{i=1}^8\mathcal{Q}(\langle X,B_i\rangle)$, where $B_i\colon 1\leq i\leq 8\}$ is an orthonormal reference of $HM_1(3)$. If $\widetilde{\Delta}$ demotes the Laplacian of the Euclidean space $HM_1(3)$, then
\[
\begin{split}
  \langle\widetilde{\Delta} X,X\rangle&=\langle\Delta X,X\rangle -\langle A^2X,X\rangle-\sum_{i=1}^2 \bigl[\lvert\hat{\sigma}(X,e_i)\rvert^2+\lvert\overline{\sigma}(X,e_i)\rvert^2\bigr]\\
                                      &=\left(K-\tfrac{1}{2}\lvert\sigma\rvert^2\right)\lvert X\rvert^2-\sum_{i=1}^2 \bigl[\lvert\hat{\sigma}(X,e_i)\rvert^2+\lvert\overline{\sigma}(X,e_i)\rvert^2 \bigr],
\end{split}
\]
where $\{e_1,e_2\}$ is an orthonormal reference on $M$. Now, from~\eqref{eq:Jacobi-operator-surface}, the Gauss equation of $\Phi$,  \eqref{eq:2ff-Berger-complex-space} and~\eqref{eq:2ff-complex-projective-euclidean} we obtain that
\begin{multline*}
  \mathcal{Q}(X)=\int_M\Bigl((\tau^2-4)\lvert X\rvert^2+ \sum_{i=1}^2\{\lvert\hat{\sigma}(X,e_i)\rvert^2+\lvert\overline{\sigma}(X,e_i)\rvert^2 \Bigr)dv\\
=\int_M\left((1-3\tau^2)\bigl(\lvert X\rvert^2+\tfrac{1-\tau^2}{\tau^2}\langle X,\xi\rangle^2\bigr)-(1-\tau^2)\bigl[\lvert(JX)^\top\rvert^2+ \tfrac{1-\tau^2}{\tau^2}\langle X,\xi\rangle^2\nu^2\bigr]\right)\, \mathrm{d}v\leq 0.
\end{multline*}
Moreover, if $\mathcal{Q}(X)=0$, then $\tau^2=1/3$, $(JX)^\top=0$ and $\langle X,\xi\rangle\nu=0$. In this case it is not difficult to get that $\nu=0$, which implies that $M$ is the Clifford surface. This is imposible because we are assuming that the genus $g\geq 5$. So $\mathcal{Q}(X)<0$ for any non-null harmonic vector field $X$ on $M$.

Suppose that $\Ind(M)=m$ and let $\{f_1,\dots,f_m\}$  the eigenfunctions of $\mathcal{L}$ corresponding to the $m$ negative eigenvalues.
If $\mathcal{H}(M)$ denotes the linear space of harmonic vector fields on $M$, we define a map
$F:\mathcal{H}(M)\rightarrow \mathbb{R}^{8m}$  by
\[
F(X)=\big( \int_Mf_1X,\dots,\int_Mf_mX).
\]
If $X\in\ker F$, then $\mathcal{Q}(X)\geq 0$, and so $X=0$. This means that $2g=\dim \hbox{Img}\, F\leq 8m$, which proves (ii).
\end{proof}

\section*{Appendix: The Tai embedding}\label{ap:complex-projective-space}

Let $HM(n+1)=\{A\in gl(n+1,\mathbb{C})\colon \overline{A}=A^t\}$ be the space of Hermitian matrices of order $n+1$, endowed with the Euclidean metric
\[
\langle A,B\rangle=\trace AB.
\]
Let $I\in HM(n+1)$ be the identity matrix. Then 
\begin{equation}\label{eq:hermitian-matrices-linear-subespace-HM1}
HM_1(n+1)=\{A\in HM(n+1)\colon \bprodesc{A}{I} = \tfrac{1}{\sqrt{2(1-\tau^2)}}\}
\end{equation} 
is an affine hyperplane of $HM(n+1)$. 

Let $\pi: \mathbb{S}^{2n+1}(\tfrac{1}{\sqrt{1-\tau^2}}) \to \mathbb{CP}^n(4(1-\tau^2))$ be the Hopf fibration of the sphere of radius $\tfrac{1}{\sqrt{1-\tau^2}}$ over the complex projective space of constant holomorphic curvature $4(1-\tau^2)$. The Tai map~\cite{Tai68} $T: \mathbb{CP}^{n}(4(1-\tau^2)) \to HM(n+1)$ is given by
\[
  T([z]) = \tfrac{\sqrt{1-\tau^2}}{\sqrt2} z^t \overline{z}, \qquad z \in \mathbb{S}^{2n+1}(\tfrac{1}{\sqrt{1-\tau^2}}) \subset \mathbb{C}^{n+1}.
\]

\begin{proposition}\label{prop:properties-embedding-complex-projective-euclidean}
  The Tai map $T$ verifies the following properties:
  \begin{enumerate}[(i)]
    \item It is an isometric embedding. 
    \item Its image is contained in the sphere $\mathbb{S}^{n^2+2n-1}(c,r)$ of $HM_1(n+1)$ with center $c = (\frac{1}{(n+1)\sqrt{2(1-\tau^2)}})I$ and radius $r = \tfrac{\sqrt{n}}{\sqrt{(n+1)2(1-\tau^2)}}$.
    \item The second fundamental form $\overline{\sigma}$ of $T: \mathbb{CP}^n(4(1-\tau^2)) \to HM(n+1)$ is parallel and $\overline{\sigma}(JX, JY) = \overline{\sigma}(X,Y)$. Moreover, for any $x, y, v, w \in T_{[z]} \mathbb{CP}^n(4(1-\tau^2))$
\begin{equation}\label{eq:2ff-complex-projective-euclidean}
  \begin{split}
    \bprodesc{\overline{\sigma}(x, y)}{\overline{\sigma}(v, w)} =\null &(1-\tau^2)\bigl[2\bprodesc{x}{y} \bprodesc{v}{w} + \bprodesc{x}{w} \bprodesc{y}{v} + \bprodesc{x}{v} \bprodesc{y}{w}\bigr]\\
    +&(1-\tau^2) \bigl[ \bprodesc{x}{Jw} \bprodesc{y}{Jv} + \bprodesc{x}{Jv} \bprodesc{y}{Jw}\bigr].
  \end{split}
\end{equation}

  \item $T: \mathbb{CP}^n(4(1-\tau^2)) \to \mathbb{S}^{n^2 + 2n - 1}(c, r)$ is a minimal embedding.
  \end{enumerate}
\end{proposition}

\begin{proof}
  This is a well-known result and it was proven in \cite[\S 1]{Ros1983} for $\mathbb{CP}^n(1)$. The result follows after minor modifications of his proof. 
\end{proof}

\end{document}